\documentclass{siamltex}

\usepackage{epsf,latexsym,amsfonts,amssymb,subfigure,mathrsfs,graphicx,stmaryrd,
amsmath}
\usepackage{color}

\newtheorem{remark}[theorem]{Remark}
\newtheorem{example}[theorem]{Example}
\numberwithin{equation}{section}

\newcommand{\mc}[1]{\mathcal{#1}}
\newcommand{\ms}[1]{\mathscr{#1}}

\newcommand{\mr}[1]{\mathrm{#1}}

\newcommand{\ov}[1]{\overline{#1}}
\newcommand{\wt}[1]{\widetilde{#1}}
\newcommand{\nm}[2]{\left\|#1\right\|_{#2}}
\newcommand{\diff}[2]{\dfrac{\pa #1}{\pa #2}}
\newcommand{\Lr}[1]{\left(#1\right)}

\def\tube{Z_{\text{bl}}}

\newcommand{\abs}[1]{\left\lvert#1\right\rvert}
\newcommand{\set}[2]{\left\{\,#1\,\mid\,#2\,\right\}}
\newcommand{\aver}[1]{\langle\,#1\,\rangle}
\def\negint{{\int\negthickspace\negthickspace\negthickspace
\negthinspace -}}
\newcommand{\wh}[1]{\widehat{#1}}

\newcommand*{\avint}{\mathop{\ooalign{$\int$\cr$-$}}}

\def\R{\mathbb R}
\def\Om{\Omega}

\def\eps{\varepsilon}
\def\ga{\gamma}
\def\pa{\partial}
\def\Ga{\Gamma}

\def\na{\nabla}
\def\Oms{\Om_{\eps}}
\def\md{\,\mathrm{d}}
\def\dx{\,\mathrm{d}x}
\def\dy{\,\mathrm{d}y}

\def\ds{\,\mathrm{d}s}
\def\dsx{\,\mathrm{d}\sigma(x)}
\def\dsxi{\,\mathrm{d}\sigma(\xi)}
\def\uu{u^\eps}
\def\hu{u^0}
\def\corr{u^{\text{cr}}}
\def\uph{\Phi_{p,\tau_{\eps}}^{\mr{MS}}}

\def\hg{\aver{g}}
\newcommand{\nn}{\nonumber}

\def\xxe{x/\eps}
\def\cut{\rho_\eps}
\def\tri{\triangle}

\begin{document}

\title
{A multiscale finite element method for oscillating Neumann problem on rough domain}%
\author{Pingbing Ming
\thanks{LSEC, Institute of Computational
  Mathematics and Scientific/Engineering Computing,
  AMSS, Chinese Academy of Sciences,
  No. 55, Zhong-Guan-Cun East Road,
  Beijing 100190, China. ({\tt mpb@lsec.cc.ac.cn}) {The work of Ming was partially
  supported by the National Natural Science Foundation of China for
Distinguished Young Scholars 11425106, and National Natural Science Foundation of China grants 91230203, and by the funds from Creative Research Groups of China through grant 11021101, and by the support of CAS NCMIS.}}
  \and{Xianmin Xu}
  \thanks{LSEC, Institute of Computational
  Mathematics and Scientific/Engineering Computing,
NCMIS,  AMSS, Chinese Academy of Sciences,
  No. 55, Zhong-Guan-Cun East Road,
  Beijing 100190, China. ({\tt xmxu@lsec.cc.ac.cn})} X. Xu acknowledges the financial support by SRF for ROCS, SEM and by NSFC grant 11571354. }
\date{\today}
\maketitle

\begin{abstract}
We develop a new multiscale finite element method for
Laplace equation with oscillating Neumann boundary
conditions on rough boundaries.
The key point is the introduction of a new boundary condition that incorporates both the microscopically geometrical and physical information of the rough boundary. Our approach applies to problems posed on domain with rough boundary as well as oscillating boundary conditions.
We prove the method has linear convergence rate in the energy norm with a weak resonance term for periodic roughness. Numerical results are reported for both periodic and nonperiodic roughness.
\end{abstract}
\begin{keywords}
Multiscale finite element method; Rough boundary; Homogenization
\end{keywords}

\begin{AMS}
 65N30; 74Q05.
\end{AMS}
\section{Introduction}
Many problems in nature and industry applications are described by
partial differential equations in domain with multiscale boundary~\cite{casado2013asymptotic, JagerMikelic:2001}.
Some of them even have oscillatory Neumann or Robin boundary conditions on the multiscale boundary~\cite{Gobbert:1998, XuWang:2010}. Theoretical study for problems with rough boundary
and oscillating boundary data mainly concerns the effective boundary conditions, which may be traced back to~\cite{KohlerPapanicalouVaradhan:1981}.
There are extensive work thereafter devoted to various topics in this field, including Poisson problem, eigenvalue problems, and Navier-Stokes equations with different types of boundary conditions; see~\cite{amirat2013effective, bucur2008asymptotic, Friedman:1997,Madureira:2007, Mikelic:2009, NeussNeussRaduMikelic:2006, NevardKeller:1997,
OleinikShamayevYosifian:1992} and the references therein. 

Compared to the extensive theoretical study, numerical methods for the rough boundary problem have been less developed, while many numerical methods have been devoted to solve elliptic problems with rough coefficients~\cite{ Babuska:1971, Babuslka:1976, BabuskaOsborn:1983, babuvska1994special,hughes1995multiscale, HouWu:1997, HouWuCai:1999, EEnquist:2003, EMingZhang:2004, OwhadiZhang:2007}. We refer to~\cite[Chapter 8]{Ebook:2011} and \cite{efendiev2009multiscale} for a review on this active field.
Only recently, a multiscale finite element method (MsFEM) was introduced to solve Laplace equation with homogeneous Dirichlet boundary value on rough domain~\cite{Madureira:2008}. The multiscale basis functions are constructed for the elements near the rough boundary by solving a cell
problem with the homogeneous Dirichlet condition on the rough edge and with linear nodal basis function as boundary condition on other edges, just as the standard MsFEM~\cite{HouWu:1997}.
However, this approach cannot be applied either to problems with non-Dirichlet boundary conditions over rough boundary, or to problems with inhomogeneous Dirichlet boundary value over the rough boundary. Therefore, one of our motivations is to develop a multiscale method for problem with oscillating boundary condition given on the rough boundary.

We introduce a new multiscale finite element method for
Laplace equation with oscillating boundary flux on the rough boundary. A Neumann boundary condition that depends on the magnitude of the flux
oscillation has to be imposed on the local cell problem posed on elements with rough edge.
When the flux oscillation is of the same order of the roughness parameter, the boundary condition contains only
the microscopical geometry of the rough boundary. Otherwise, one has to
incorporate both the microscopical geometry of the rough boundary and the flux oscillation into the boundary condition. Such multiscale basis function coincides with the linear nodal
basis functions for elements without a rough edge. This method is H$^1$-conforming,
with degrees of freedom at the mesh nodes and the basis functions are solved over
the elements near the rough boundary and can be computed off line. For periodic roughness, we prove that our method has optimal convergence
rate in the energy norm besides a weak resonance term. The method also applies to problems
with non-periodic roughness as demonstrated by the numerical experiments. The proof is based on certain homogenization
results for Neumann rough boundary value problems, which refine the corresponding results in~\cite{Friedman:1997} by clarifying the dependence of the error bounds on the domain size.
Our convergence results require that
the right-hand side function $f\in H^1$. However, numerical experiments show that the optimal convergence order is retained for even rougher L$^2$ right-hand side function.

The novelty of the proposed method is that both the rough boundary
and the oscillating flux are considered and no structure is assumed for the oscillations.
The method can be naturally generalized to the inhomogeneous Dirichlet boundary value problem over rough domain. It is also possible
to combine the proposed method with the standard MsFEM to deal with the problems
with oscillatory coefficients and oscillatory boundary data. 

The so-called composite finite elements has been successfully applied to
solve boundary value problems over
complicated domain~\cite{HackbuschSauter:1997, HackbuschSauter:1999, peterseim2011}. The basic idea of this method
is to incorporate the geometrical complexity of the domain into the basis function, while there is no local cell problems. Optimal convergence rate has been achieved, which is independent of the geometrical structure of the domain~\cite{HackbuschSauter:1997, NicaiseSauter:2006, rech2006}. In particular, homogeneous Neumann boundary value problems have been studied in~\cite{HackbuschSauter:1997} and~\cite{NicaiseSauter:2006}.  In contrast to composite finite elements, the multiscale basis function is constructed by solving local problem that contains both the geometry complexity and the oscillation of the boundary flux.

A multiscale method has been developed for elliptic equation with homogeneous boundary condition on complicated domains very recently in~\cite{Elfverson15}. The method is based on the localized orthogonal
decomposition (LOD) technique developed in~\cite{Malqvist11, Elfverson13, Malqvist14, Henning14, Malqvist15}.
The method is quite general and has optimal convergence order.
The difference between LOD method and the proposed method is the way in dealing with the multiscale basis function near the rough boundary.  In the LOD method, cell problems
are solved in several layers of elements near the boundary. In our method, the multiscale basis functions are solved only in cells with rough boundary. When the underlying scales of the problem are well-separated, MsFEM would be less expensive. Furthermore, we study the inhomogeneous oscillating flux on the boundary while only homogeneous boundary conditions have been treated in~\cite{Elfverson15}. In addition, we note that the heterogeneous multiscale method and MsFEM have been employed to solve partial differential equations on a rough surface in~\cite{abdulle2005heterogeneous} and~\cite{efendiev2013multiscale}, respectively, however, the authors have not  dealt with the problems studied in this paper.

Finally, the problem with complicated domain can be discretized by adaptive finite element method~\cite{Verfurth96}.
{However, the mesh size of the adaptive method must be much smaller than the characteristic length scale of the roughness for the sake of resolution, which would results in a linear algebraic system with large condition number. In contrast to the adaptive method, the condition number of the resulting linear system of the proposed method is proportional to
$\mc{O}(h^{-2})$ as demonstrated in the numerical experiment, while the coarse grid size $h$ is not necessarily smaller than the roughness length scale.

The structure of the paper is as follows. In Section 2,
we describe the model problem and introduce the multiscale finite element method.
In Section 3, we revisit the homogenization results for a Possion equation with an oscillatory Neumann boundary condition over rough domain. In Section 4, we estimate the convergence rate in energy norm of the proposed method. Numerical examples are illustrated in the last section.

\section{The Model Problem and Multiscale Finite Element Method}
Let $\Oms\subset \R^2$ be a bounded domain with boundary $\pa\Oms$, a part of which is rough and denoted as
$\Ga_\eps$, where $\eps$ is a small parameter that characterizes the roughness of $\Gamma_\eps$.
We consider a model problem with Neumann boundary conditions on $\Ga_\eps$: Given the source term $f$ and the flux $g_\eps$ that is oscillatory, we find $\uu$ satisfying
\begin{equation}\label{eq:modelproblem}
\left\{
\begin{aligned}
-\Delta \uu& =f(x),\qquad  &&\hbox{in $\Oms$,} \\
\uu &=0,\qquad  &&\hbox{on $\Ga_D$,} \\
\diff{\uu}{n}&=g_{\eps}(x),\qquad    &&\hbox{on $\Ga_\eps$,}
\end{aligned}\right.
\end{equation}
where $\Ga_D=\partial\Oms\setminus\Ga_\eps$.

For any measurable subset $D$ of $\Oms$, we define
\[
V(D)=\set{v\in H^1(D)}{v|_{\pa D\setminus\Gamma_{\eps}}=0}.
\]
Here $H^1(D)$ is the standard Sobolev space, and the notations and definitions for
Sobolev spaces can be found in~\cite{AdamsFournier:2003}. To clarify the dependence of the roughness parameter $\eps$, we denote $\uu$ the solution
of Problem~\eqref{eq:modelproblem}, whose weak form is:
Find $\uu\in V(\Oms)$ such that
\begin{equation}\label{eq:weakform}
a(\uu,v)=(f,v)+(g_\eps,v)_{\Ga_\eps} \qquad \text{for all\quad} v\in V(\Oms),
\end{equation}
where
\[
a(\uu,v):=\int_{\Oms}\na\uu\cdot \na v\dx, \quad (f,v):=\int_{\Oms}fv\dx, \quad
(g_\eps,v)_{\Gamma_{\eps}}:=\int_{\Gamma_{\eps}}g_\eps v\ds.
\]

We triangulate $\Oms$ by a shape regular mesh $\mc{T}_h$ in the sense of~\cite{Ciarlet:1978},
with element $\tau \in\mc{T}_h$ be either a triangle or a
quadrilateral, where $h=\max_{\tau\in\mathcal{T}_h}h_{\tau}$ with $h_{\tau}$ the diameter of $\tau$,
and $\ms{S}(\tau)$ is a suitable index set for nodes in $\tau$. We assume that an element
near $\Ga_{\eps}$ has at most one rough edge on the rough boundary and denote such elements by $\tau_{\eps}$; see Fig.~\ref{fig:partition}. For triangular mesh, there are elements that may have only one node
on the rough boundary. In Section 5, we shall give more details on the triangulation.
 \begin{figure}[htbp]  \hspace{-1cm}
\resizebox{!}{5cm}
{\subfigure[rectangular mesh]{%
\includegraphics[height=5.0cm]{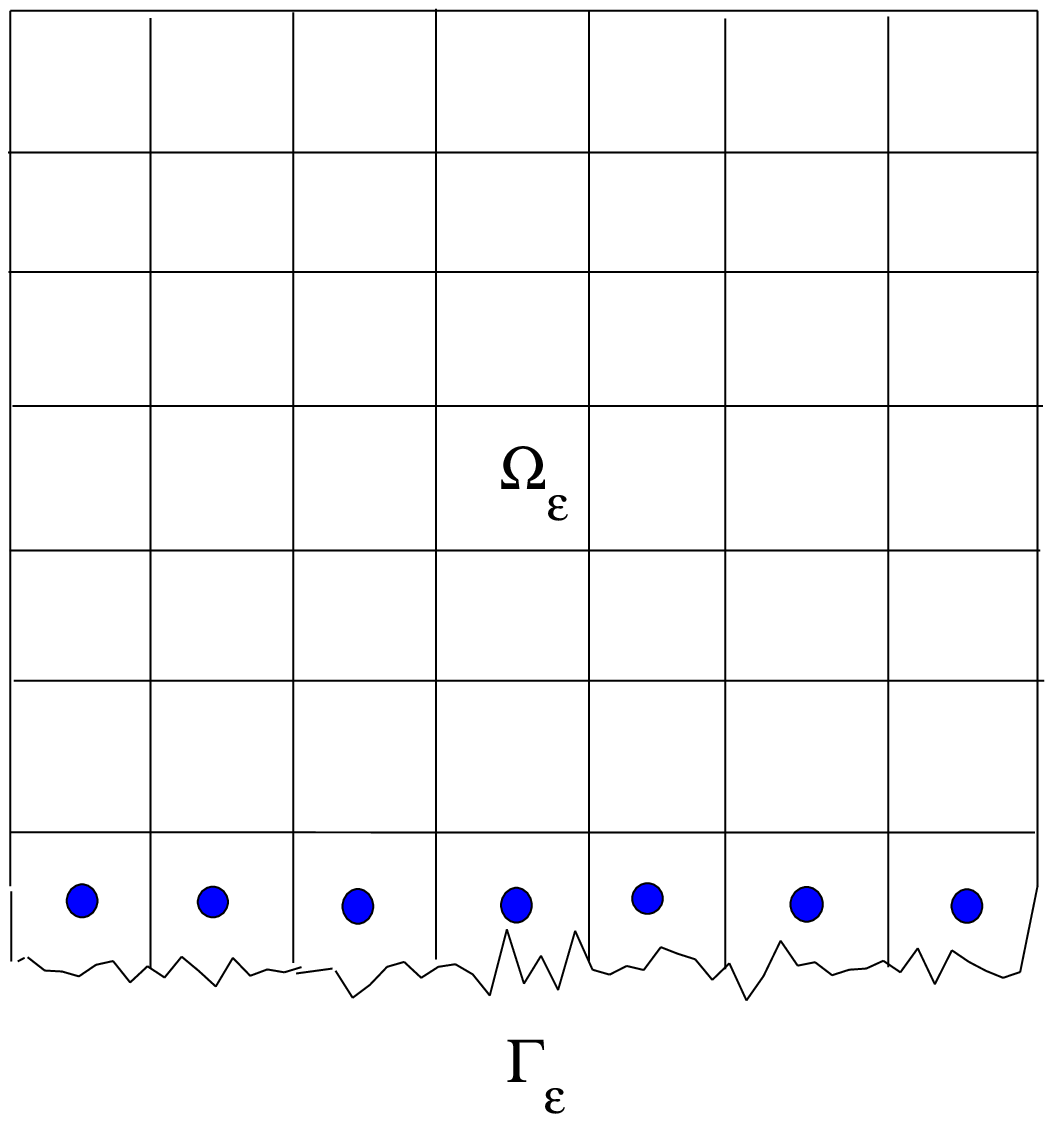}}}%
  \subfigure[triangular mesh]{%
  \includegraphics[height=5cm]{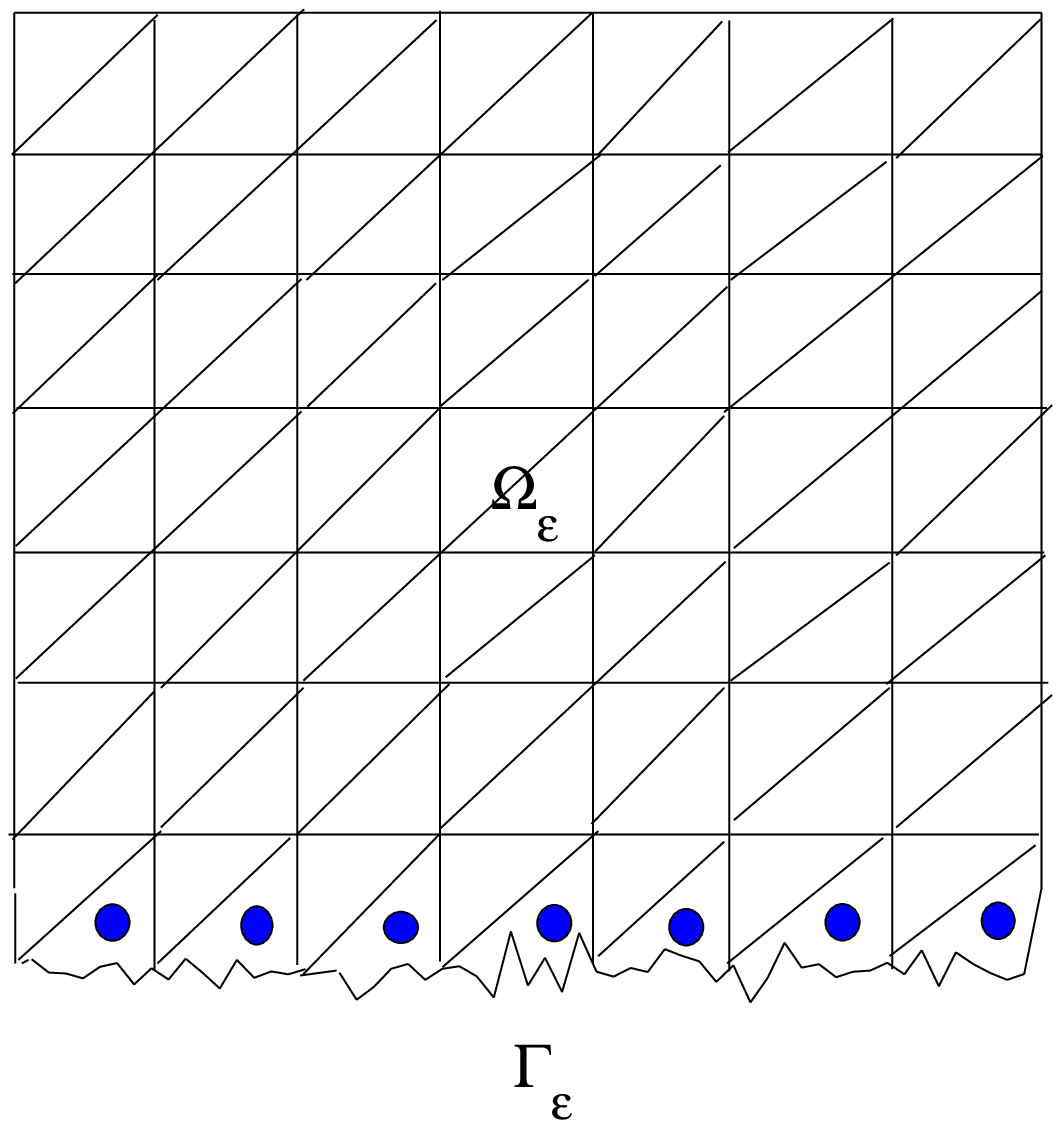}}%
\caption{Examples for the triangulation of the domain $\Omega_{\eps}$.}
\label{fig:partition}
\end{figure}

For any measurable subset $D$ of $\Oms$, we define a localized version of $a$ as
\[
a_D(v,w):= \int_D\na v\cdot\na w\dx\qquad v,w\in H^1(D).
\]

For each $p\in\ms{S}(\tau)$ we construct nodal basis functions $\Phi_{p}^{\mr MS}$, whose restriction to each element $\tau$ is denoted by $\Phi_{p,\tau}^{\mr MS}$, which satisfies
\begin{equation}\label{eq:cellpb}
a_{\tau}(\Phi_{p,\tau}^{\mr MS},v)=(\theta_{p,\tau},v)_{\partial\tau\cap\Gamma_{\eps}} \quad \hbox{for all } v\in V(\tau),
\end{equation}
and $\Phi_{p}^{\mr MS}$ is supplemented with the boundary condition
\begin{equation}\label{eq:basisbc}
\Phi^{\mr MS}_{p,\tau}=\phi_{p,\tau}\quad\text{on\quad}\pa\tau\setminus\Gamma_{\eps},
\qquad \Phi^{\mr MS}_{p,\tau}(x_q)=\delta_{p,q}\quad \text{for
all\quad}p,q\in\ms{S}(\tau),
\end{equation}
where $\phi_{p,\tau}$ is the restriction of the standard linear nodal basis function $\phi_p$ on $\tau$.

The flux $\theta_{p,\tau_\eps}$ is defined as follows. If $\tau$ has no edge on $\Ga_\eps$, then
we let $\theta_{p,\tau}=0$. Problem~\eqref{eq:cellpb} changes to a Dirichlet boundary value problem with a unique solution $\Phi_{p,\tau}^{\mr MS}=\phi_{p,\tau}$, i.e., the multiscale basis function coincides with
the linear basis function. If $\tau_\eps$ has one edge on $\Ga_\eps$, then
\begin{equation}\label{eq:cellbcd}
\theta_{p,\tau_\eps}(x):=\left\{
\begin{aligned}
\dfrac{\pa_n\phi_{p,\tau_0}}{r}&\qquad\text{if\;}\|g_\eps-\aver{g_\eps}\|_{L^\infty(\pa\tau_\eps)}\le C\eps,\\
\dfrac{\pa_n\phi_{p,\tau_0}}{r}\dfrac{g_\eps(x)}{\aver{g_\eps}}&\qquad\text{otherwise.}
\end{aligned}\right.
\end{equation}
In this case, the local problem has mixed boundary conditions.
Here the parameter $r=\abs{s_{\eps}}/\abs{s_0}$ 
with $s_\eps$ being the rough edge of $\tau_\eps$, while $s_0$ being the homogenized rough edge, and $n$ is the unit outer normal of $s_0$, $\tau_0$ is the homogenized element of $\tau_\eps$ and $\aver{g_\eps}=\avint_{s_{\eps}} g_\eps$ is the mean of $g_\eps$ over $s_\eps$.

The flux $\theta_{p,\tau_\eps}$ defined in~\eqref{eq:cellbcd}$_2$ contains both geometrical and physical information of the boundary.
If the flux has no oscillations, i.e., $g_\eps=\aver{g_\eps}$, then~\eqref{eq:cellbcd}$_2$ changes to~\eqref{eq:cellbcd}$_1$, and the flux does not contain the physical information any more. Furthermore, if the boundary is also flat, i.e., $r=1$, then the unique solution of the cell problem is the linear nodal basis functions, and the method automatically changes to the standard finite element method.

The bound $C\varepsilon$ in~\eqref{eq:cellbcd} is a threshold for determining whether the physical information should be incorporated into
the cell problem. Roughly speaking, if $\|g_\eps-\aver{g_\eps}\|_{L^\infty(\pa\tau_\eps)}$ is as small as $\mc{O}(\varepsilon)$, then
we need not any physical information but the geometrical information of the rough boundary. Otherwise, the physical information should be incorporated into the cell problem. This is consistent with our intuition as seen from the example below. 
If there is no information on $\eps$, we can use~\eqref{eq:cellbcd}$_2$ whenever $\aver{g_\eps}\neq 0$.
\begin{example}
If $g_\eps(x)=\eps \sin(x/\eps)$, it is clear that $\|g_\eps-\aver{g_\eps}\|_{L^\infty}=\varepsilon$, then we may use~\eqref{eq:cellbcd}$_1$. On the other hand,
if $g_\eps(x)=1+ \sin(x/\eps)$, then $\|g_\eps-\aver{g_\eps}\|_{L^\infty}=1$, and we have to use~\eqref{eq:cellbcd}$_2$. There are some special cases beyond~\eqref{eq:cellbcd}, e.g., $g_\eps(x)= \sin(x/\eps)$ so that $\|g_\eps\|_{\infty}=1$ but $\aver{g_\eps}=0$.
In this case, the method works as well if we decompose $g_\eps$ as $g_\eps= g_1+g_2$ with $g_1{:}=1$ and $g_2{:}= \sin(x/\eps)- 1$ and split Problem~\eqref{eq:modelproblem} into two problems with boundary conditions $g_1$ and $g_2$, respectively. We would like to emphasis that this splitting technique does not apply to the nonlinear problems because
the flux is not well-defined when $\aver{g_\eps}=0$ and $\|g_\eps\|_{L^\infty}=\mc{O}(1)$. Nevertheless, it may be directly applied to
the corresponding linearized problems.
\end{example}

Under the conditions~\eqref{eq:cellpb},~\eqref{eq:basisbc} and~\eqref{eq:cellbcd}, we have, for all
$\tau\in\mathcal{T}_h$,
\[
\sum_{p\in\ms{S}(\tau)}\Phi_{p,\tau}^{\mr MS}=1.
\]
The basis function $\phi_{p}^{\mr MS}$ is continuous across the element boundary so that
\[
V_{h}{:}=\text{span}\set{\phi_{p}^{\mr MS}}{ p\in\ms{S}(\Oms)}\subset V(\Oms).
\]
The MsFEM approximation of Problem~\eqref{eq:modelproblem} is
to find $u_h\in V_h$ such that
\begin{equation}\label{eq:MsFE}
a(u_h,v)=(f,v)+(g_\eps,v)_{\Ga_\eps}\qquad\text{for all } v\in V_h.
\end{equation}
This is a conforming method, and the existence and uniqueness of the solution follow
from Lax-Milgram theorem. Moreover, we have 
\begin{equation}\label{eq:cea}
\nm{\na(\uu-u_h)}{L^2(\Oms)}=\inf_{v\in V_h}\nm{\na(\uu-v)}{L^2(\Oms)}.
\end{equation}
The error estimate now boils down to the interpolate error estimate, which will be the
focus of the later sections.

The MsFEM problem~\eqref{eq:MsFE} has $\mc{O}(h^{-2})$ freedoms in two dimension. To calculate each multiscale basis, we need $\mc{O}(\wt{h}^{-2})$ freedoms with $\wt{h}$ the mesh size of the local cell problem. The number of the cell problem is $\mc{O}(h^{-1})$. The overall complexity of the proposed method is of $\mc{O}(h^{-2}+h^{-1}\wt{h}^{-2})$. Note that $\wt{h}\simeq M^{-1}h$, the complexity is of $\mc{O}(M^2h^{-3})$. The cell problems are independent of each other, and could be solved in parallel.
%
\begin{remark}
The proposed method can be generalized
to the problem with oscillatory inhomogeneous Dirichlet boundary conditions on the
rough surface. We assume $u^\eps=g_\eps$ on $\Gamma_\eps$. The MsFEM basis functions $\Phi_{p,\tau}^{\mr MS}$ satisfy
\begin{equation}\label{eq:cellpbDirich}
a_{\tau}(\Phi_{p,\tau}^{\mr MS},v)=0\quad \hbox{for all } v\in H^{1}_0(\tau),
\end{equation}
and is supplemented with the boundary condition: $\Phi^{\mr MS}_{p,\tau}=\theta_{p,\tau_\eps}(x)$ on the rough edge $\partial\tau\cap\Gamma_\eps$ and $\Phi^{\mr MS}_{p,\tau}=\phi_{p,\tau}$ on $\pa\tau\backslash\Ga_\eps$, where
\begin{equation}\label{eq:cellbcdDirich}
\theta_{p,\tau_\eps}(x):=\left\{
\begin{aligned}
\dfrac{\phi_{p,\tau_0}}{r(x)}
&\qquad\text{if\;}\|g_\eps-\aver{g_\eps}\|_{L^\infty(\pa\tau_\eps)}\le C\eps,\\
\dfrac{\phi_{p,\tau_0}}{r(x)}
\dfrac{g_\eps(x)}{\aver{g_\eps}(x)}&\qquad\text{otherwise,}
\end{aligned}\right.
\end{equation}
The detailed analysis of the MsFEM for inhomogeneous Dirichlet boundary value
problem will be addressed in a future paper.
\end{remark}
\section{Error Estimate for the Homogenization Problem}
In this section, we revisit some homogenization results for Problem~\eqref{eq:modelproblem}, which have been established in~\cite{Friedman:1997}, while we clarify the dependence of the estimates on the domain size, which is crucial for studying the accuracy of the proposed method. Our approach is different from that in~\cite{Friedman:1997} for estimate of
the first order approximation.

We assume that $\Oms$ is given by
\begin{equation}\label{eq:square}
\Oms{:}=\set{x\in\R^2}{0<x_1<1,\eps \gamma(x_1/\eps)<x_2<1},
\end{equation}
and the oscillating bottom boundary $\Ga_\eps$ is given by
\[
\Ga_\eps=\set{x\in\ov{\Om}_\eps}{0<x_1<1, x_2=\eps\gamma(x_1/\eps)}
\]
with $\gamma$ a positive smooth $1-$periodic function. We assume that $g_\eps(x_1)=g({x_1}/{\eps})$ with $g$
a smooth $1-$periodic function and satisfying
\begin{equation}\label{e:assump}
\nm{g-\hg}{L^{\infty}(\Sigma)}\leq C(\abs{\hg}+\eps),
\end{equation}
where $\Sigma=\set{\xi\in\R^2}{0<\xi_1<1,\xi_2=\gamma(\xi_1)}$ with $\hg$ the mean of $g$ over $\Sigma$:
\[
\hg =\dfrac{1}{r}\int_0^1 g(t)[1+(\gamma'(t))^2]^{1/2}\md\,t \quad\hbox{and} \quad r=\int_0^1 [1+ (\gamma'(t))^2]^{1/2}\md\,t.
\]
\begin{figure}[htbp]  
  \resizebox{!}{6cm}
  {\includegraphics[height=6.0cm]{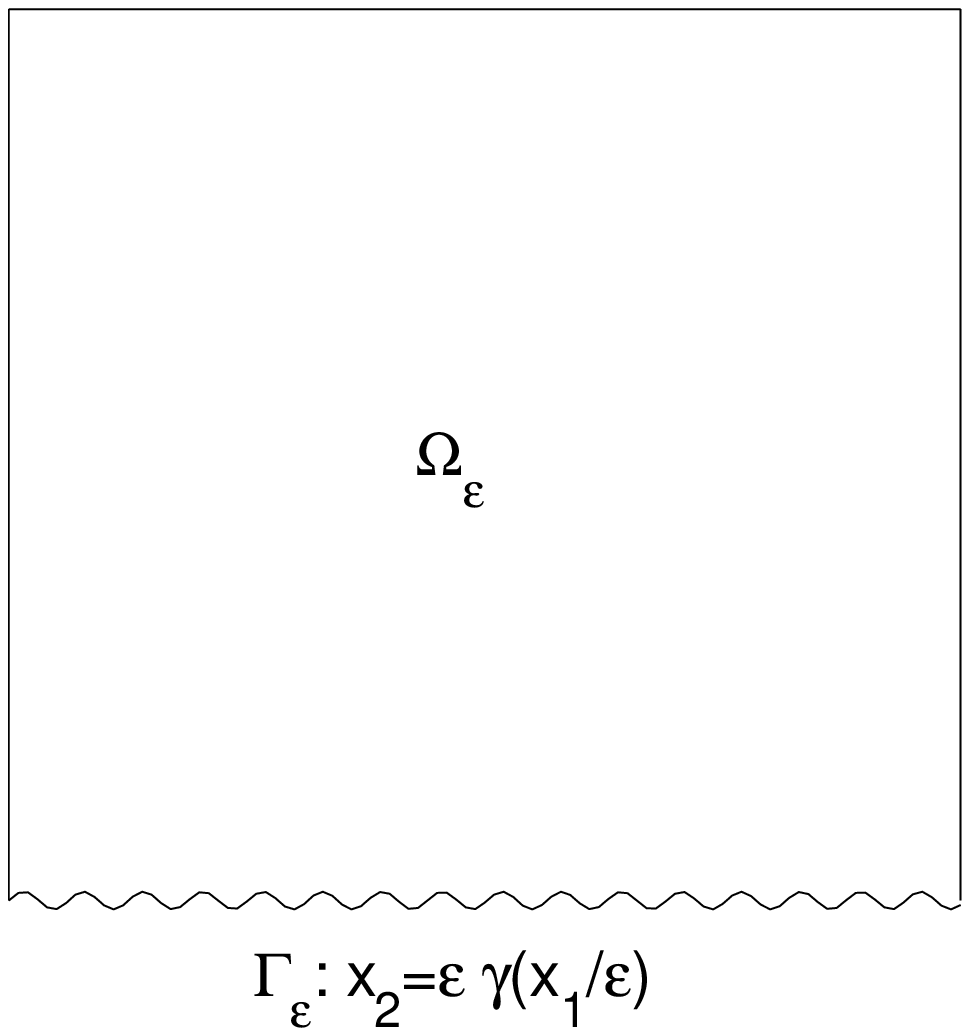}
  \hspace{-1cm}
  \includegraphics[height=6cm]{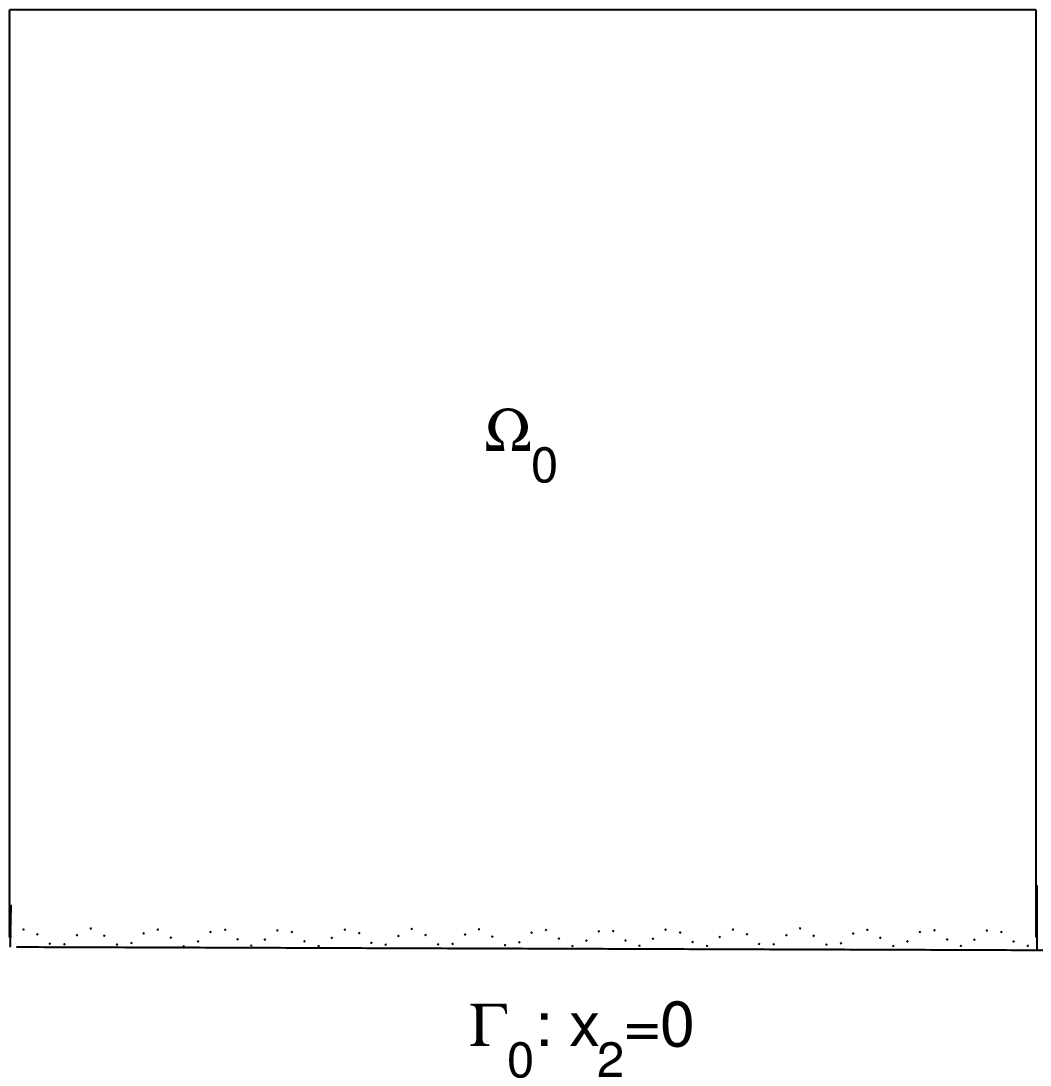}}
  \vspace{-1.0cm}
\caption{The domains $\Omega_{\eps}$ and $\Omega_0$.}
\label{fig:domain4}
\end{figure}
\begin{remark}\label{rem:omega}
The assumption $\gamma\geq 0$ ensures $\Oms\subset\Om_0$. This may make the presentation slightly
simpler.
All the results in this section remain valid for general $\ga$, we refer to~\cite[\S 8]{Friedman:1997}
and~\cite{NeussNeussRaduMikelic:2006} for related discussions.
\end{remark}
\subsection{The zeroth order approximation}
Let
\[
\Om_0=\set{x\in\R^2}{0<x_1<1, 0<x_2<1}\quad\text{and\quad}\Ga_0=\pa\Om_0\cap\set{x\in\R^2}{x_2=0}.
\]
Define \(\Ga_D{:}=\pa\Om_0\setminus\Ga_{0}\) and
\[
V(\Om_0)=\set{v\in H^1(\Om_0)}{v|_{\Ga_D}=0}.
\]
The zeroth order approximations $\hu$ of $\uu$  is such that $\hu\in V(\Om_0)$ and
\begin{equation}\label{eq:dirichlet0}
a_{\Om_0}(\hu,v)=(f,v)_{\Om_0}+(r\hg,v)_{\Ga_0}\qquad \hbox{for all } v\in V(\Om_0).
\end{equation}

We start with an extension result that will be frequently used later on, which
slightly refines that in~\cite[Appendix A]{Mclean:2000}.
\begin{lemma}\label{lema:extension}
There exists an extension operator $E: H^1(\Oms)\to H^1(\Om_0)$ such that, for any $\phi\in H^1(\Oms)$,
\begin{equation}\label{eq:extension}
\nm{E\phi}{H^1(\Om_0)}\le\sqrt{2+2A^2+2A\sqrt{1+A^2}}\nm{\phi}{H^1(\Oms)},
\end{equation}
where $A=\nm{\gamma^{\prime}}{L^\infty(0,1)}$.
\end{lemma}

\begin{proof}
For any $\phi\in H^1(\Oms)$, we define $E\phi$ by reflection with respect to $\Ga_\eps$.
\[
E\phi(x)=\left\{
\begin{aligned}
\phi(x)&\qquad x\in\Oms,\\
\phi(x_1,2\eps \gamma(x_1/\eps)-x_2)&\qquad x\in\Om_0\backslash\Oms.
\end{aligned}\right.
\]
A direct calculation gives
\[
\pa_{x_1}E\phi(\mr x)=\left\{
\begin{aligned}
\pa_{x_1}\phi(\mr x)&\qquad x\in\Oms,\\
\pa_{x_1}\phi+2\pa_{x_2}\phi \gamma^{\prime}(x_1/\eps)&\qquad x\in\Om_0\backslash\Oms.
\end{aligned}\right.
\]
and
\[
\pa_{x_2}E\phi(x)=\left\{
\begin{aligned}
\pa_{x_2}\phi(x)&\qquad x\in\Oms,\\
-\pa_{x_2}\phi(x_1,2\eps \gamma(x_1/\eps)-x_2)&\qquad x\in\Om_0\backslash\Oms.
\end{aligned}\right.
\]
Note the module of the Jacobian of the substitution
\[
(x_1,x_2)\mapsto(x_1,2\eps \gamma(x_1/\eps)-x_2)
\]
is equal to one. Observe that the inequality
\[
\nm{E\phi}{L^2(\Om_0)}^2\le 2\nm{\phi}{L^2(\Oms)}^2
\]
holds, where we have used
\begin{align*}
\nm{E\phi}{L^2(\Om_0\backslash\Oms)}^2&=\int_0^1\int_0^{\eps \gamma(x_1/\eps)}\abs{\phi(x_1,2\eps \gamma(x_1/\eps)-x_2)}^2\dx\\
&=\int_0^1\int_{\eps \gamma(x_1/\eps)}^{2\eps \gamma(x_1/\eps)}\abs{\phi(x_1,x_2)}^2\dx\\
&\le\nm{\phi}{L^2(\Oms)}^2.
\end{align*}
Next, we estimate $\pa_{x_1}E\phi$ and $\pa_{x_2}E\phi$.
\begin{align*}
\nm{\pa_{x_1}E\phi}{L^2(\Om_0)}^2&\le\nm{\pa_{x_1}\phi}{L^2(\Oms)}^2+\nm{\pa_{x_1}\phi
+2\ga^{\prime}\pa_{x_2}\phi}{L^2(\Oms)}^2\\
&\le (2+t)\nm{\pa_{x_1}\phi}{L^2(\Oms)}^2+4A^2(1+1/t)\nm{\pa_{x_2}\phi}
{L^2(\Oms)}^2
\end{align*}
with any $t>0$, and
\[
\nm{\pa_{x_2}E\phi}{L^2(\Om_0)}^2\le 2\nm{\pa_{x_2}\phi}{L^2(\Oms)}^2.
\]
The above three estimates imply
\begin{align*}
\nm{E\phi}{H^1(\Om_0)}^2&\le (2+t)\nm{\pa_{x_1}\phi}{L^2(\Oms)}^2+\Lr{2+
4A^2(1+1/t)}\nm{\pa_{x_2}\phi}{L^2(\Oms)}^2+2\nm{\phi}{L^2(\Oms)}^2.
\end{align*}
Hence,
\[
\nm{E\phi}{H^1(\Om_0)}^2 \le \max\{2+t,2+
4A^2(1+1/t)\}\nm{\phi}{H^1(\Oms)}^2.
\]
Optimizing the maximum with respect to parameter $t$, we obtain the maximum
is minimal if \(t=2A^2+2\sqrt{A^2+A^4}\). This completes the proof.
\end{proof}

To estimate the error between $\hu$ and $\uu$, we need an auxiliary result~\cite[Lemma 1.5, p.7]{OleinikShamayevYosifian:1992}. The present form can be found in~\cite[inequality (17) in Lemma 10]{NicaiseSauter:2006}.
\begin{lemma}\label{lema:strip}
Let $\Om$ be a Lipschitz domain, and $S_\eps{:}=\set{x\in\Om}{\text{dist}(x,\pa\Om)\le\eps}$, then for
any $1/2<\kappa\le\mu\le 1$ and $v\in H^\mu(\Om)$, there holds
\begin{equation}\label{eq:strip}
\nm{v}{L^2(S_\eps)}\le C\Lr{\sqrt\eps\nm{v}{H^{\kappa}(\Om)}+\eps^\mu\nm{v}{H^\mu(\Om)}},
\end{equation}
where $C>0$ is a constant independent of $\eps$.
\end{lemma}
\begin{lemma}\label{prop:zeroth}
Let $\uu$ and $\hu$ be the solutions to Problems~\eqref{eq:modelproblem} and~\eqref{eq:dirichlet0}, respectively. Then
\begin{equation}\label{eq:zeroth}
\nm{\na(\uu-\hu)}{L^2(\Oms)}\le  C\sqrt\eps(\nm{
f}{L^2(\Om_0)} + \nm{\na\hu}{H^1(\Om_0)}+
\|g_\eps \|_{L^2 (\Gamma_{\eps})}).
\end{equation}
\end{lemma}

\begin{proof}
Denote $e =\uu -\hu \in V(\Omega_\eps)$. For any $\phi\in V(\Omega_\eps)$, we have
\begin{align}
&\int_{\Oms}\nabla e\nabla \phi\dx=-\int_{\Om_0\setminus\Oms}(f\phi-\na\hu\cdot\na\phi)\dx
+\int_{\Ga_\eps} g_\eps \phi\md\sigma(x)-\int_{\Ga_0} r \hg \phi \dx_1 \nn\\
=&-\int_{\Om_0\setminus\Oms} (f\phi-\na\hu\cdot\na\phi) \dx\nn\\
&+\int_0^1 \Lr{g(x_1/\eps)\phi(x_1,\eps \gamma(x_1/\eps))[1+( \gamma^{\prime}(x_1/\eps))^2]^{1/2}-r\hg \phi (x_1,0) }\dx_1\nn \\
=&I_1+I_2.\nn
\end{align}
Using~\eqref{eq:strip}, we bound $I_1$ as
\begin{align}\label{e:estI1}
\abs{I_1}&\le\nm{f}{L^2(\Om_0\setminus\Oms)}\nm{\phi}{L^2(\Om_0\setminus\Oms)}
+\nm{\na\hu}{L^2(\Om_0\setminus\Oms)}\nm{\na\phi}{L^2(\Om_0\setminus\Oms)}\nn\\
&\le C\sqrt\eps\nm{f}{L^2(\Om_0)}\nm{\phi}{H^1(\Om_0)}
+C\sqrt\eps\nm{\na\hu}{H^1(\Om_0)}\nm{\na\phi}{L^2(\Om_0)}\nn\\
&\le C\sqrt\eps\Lr{\nm{f}{L^2(\Om_0)}+\nm{\na\hu}{H^1(\Om_0)}}\nm{\na\phi}{L^2(\Om_0)},
\end{align}
where in the last step we have used {\em Poincar\'e's inequality} for $\phi$.

Denote by $p^\eps(x_1)=p(x_1/\eps)$ with $p(t)=g(t)[1+\gamma^{\prime}(t)^2]^{1/2}$, we have
\[
I_2=\int_0^1\Lr{p^\eps(x_1)\phi(x_1,\eps \gamma(x_1/\eps))-\aver{p}\phi(x_1,0)}\dx_1,
\]
where $\aver{p}$ denotes the average of $p$ and $\aver{p}=r\aver{g}$, i.e.,
\(
\aver{p}=\int_0^1g(t)[1+\gamma^{\prime}(t)^2]^{\frac12}\md\,t.
\)

It is clear to see
\[
I_2=\int_0^1p^\eps(x_1)\Lr{\phi(x_1,\eps \gamma(x_1/\eps))-\phi(x_1,0)}\dx_1
+\int_0^1\Lr{p^\eps-\aver{p}}\phi(x_1,0)\dx_1.
\]
A direct calculation gives that
\begin{align*}
\abs{\int_0^1p^\eps(x_1)\Lr{\phi(x_1,\eps \gamma(x_1/\eps))-\phi(x_1,0)}\dx_1}
&=\abs{\int_0^1p^\eps(x_1)\int_0^{\eps \gamma(x_1/\eps)}\diff{\phi}{x_2}\dx_2\dx_1}\\
&\le\sqrt\eps\nm{\na\phi}{L^2(\Om_0\backslash\Oms)}\nm{p^\eps \gamma^{1/2}}{L^2(\Ga_0)}\\
&\le C\sqrt\eps\nm{\na\phi}{L^2(\Om_0)}\nm{g_{\eps}}{L^2(\Gamma_\eps)},
\end{align*}
where we have used
\[
\nm{p^\eps \gamma^{1/2}}{L^2(\Ga_0)}^2
\le A\nm{\gamma}{L^\infty(0,1)}\nm{g_{\eps}}{L^2(\Gamma_\eps)}^2.
\]

Denote by $s_0=0<s_1=\eps<\cdots<s_N=N\eps=1$, using the fact that
\[
\aver{p}=\aver{p^\eps}_i=\negint_{s_i}^{s_{i+1}}p^\eps(x_1)\dx_1,
\]
we decompose the second term into
\begin{align*}
\int_{\Ga_0}\Lr{p^\eps(x_1)-\aver{p}}\phi(x_1,0)\dx_1
&=\sum_{i=0}^{N-1}\int_{s_i}^{s_{i+1}}p^\eps(x_1)\Lr{\phi(x_1,0)-\aver{\phi}_i}\dx_1,
\end{align*}
where
\(
\aver{\phi}_i=\negint_{s_i}^{s_{i+1}}\phi(x_1,0)\dx_1.
\)
Using {\em Poincar\'e's inequality}, we obtain
\begin{align*}
\abs{\int_{\Ga_0}\!\!\Lr{p^\eps(x_1)\!-\!\aver{p}}\phi(x_1,0)\dx_1}&\le\sum_{i=0}^{N-1}
\nm{p^\eps}{L^2(s_i,s_{i+1})}\nm{\phi-\aver{\phi}_i}{L^2(s_i,s_{i+1})}\\
&\le C\sqrt\eps\sum_{i=0}^{N-1}\nm{p^\eps}{L^2(s_i,s_{i+1})}\nm{\phi}{H^{1/2}(s_i,s_{i+1})}\\
&\le C\sqrt\eps\!\Lr{\sum_{i=0}^{N-1}\!\nm{p^\eps}{L^2(s_i,s_{i+1})}^2\!}^{\!\!\!1/2}
\!\!\!\Lr{\sum_{i=0}^{N-1}\!\nm{\phi}{H^{1/2}(s_i,s_{i+1})}^2\!}^{\!\!\!1/2}\\
&\le C\sqrt\eps\nm{g_\eps}{L^2(\Gamma_{\eps})}\nm{\phi}{H^{1/2}(\Ga_0)}\\
&\le C\sqrt\eps\nm{g_\eps}{L^2(\Gamma_{\eps})}\nm{\phi}{H^1(\Om_0)},
\end{align*}
where we have used the fact that
\(
\sum_{i=0}^{N-1}\nm{p^\eps}{L^2(s_i,s_{i+1})}^2\le(1+A) \nm{g_\eps}{L^2(\Gamma_{\eps})}^2.
\)
This implies
\begin{equation}\label{eq:estI2}
\abs{I_2}\le C\sqrt\eps\nm{\phi}{H^1(\Om_0)}
\nm{g_\eps}{L^2(\Gamma_{\eps})}.
\end{equation}

Using the extension result~\eqref{eq:extension}, we have
\[
\nm{u^\eps-u^0}{H^1(\Om_0)}\le C\nm{u^\eps-u^0}{H^1(\Oms)},
\]
where $C$ only depends on $\|\gamma^{\prime}\|_{L^\infty(0,1)}$. This inequality together with~\eqref{e:estI1} and~\eqref{eq:estI2}
implies~\eqref{eq:zeroth}.
\end{proof}

The above lemma shows that $\hu$ approximates to $\uu$ in $H^1$ seminorm
with rate $\mc{O}(\sqrt\eps)$.
The convergence rate is inadequate in many applications. We step to the
first order approximation in the next part.
\subsection{Some auxiliary problems}
To find the next order approximation of Problem~\eqref{eq:modelproblem}, we define a semi-infinite tube
as
\[
\tube{:}=\set{\xi \in\R^2}{0<\xi_1<1,\xi_2>\gamma(\xi_1)}
\]
with a curved boundary
\(
\Sigma{:}=\set{\xi \in\R^2}{0<\xi_1<1,\xi_2=\gamma(\xi_1)}.
\)

Three auxiliary problems are defined as follows. Let $\beta_0, \beta_1$ and $\beta_2$ be
three unknown functions posed on $\tube$, which are periodic
in $\xi_1$ with period $1$ and satisfy
\begin{equation}\label{eq:aux1}
\left\{\begin{aligned}
-\triangle_{\xi}\beta_0&=0,\qquad &&\hbox{in }\tube, \\
\diff{\beta_0}{n}&=  g(\xi_1)- \hg,\quad  &&\hbox{on $\Sigma$}, \\
\lim_{\xi_2\to\infty}\beta_0 &=0,\quad&&
\end{aligned}\right.
\end{equation}
and
\begin{equation}\label{eq:aux2}
\left\{\begin{aligned}
-\triangle_{\xi}\beta_1& =0, \qquad&&\hbox{in }\tube, \\
\diff{\beta_1}{n}&=-\dfrac{\gamma'(\xi_1)}{[1+(\gamma'(\xi_1))^2]^{1/2}},\quad
 &&\hbox{on $\Sigma$,} \\
\lim_{\xi_2\to\infty}\beta_1&=0,\quad&&
\end{aligned}\right.
\end{equation}
and
\begin{equation}\label{eq:aux3}
\left\{\begin{aligned}
-\triangle_{\xi}\beta_2&=0, \qquad&&\hbox{in }\tube, \\
\diff{\beta_2}{n}&=\dfrac{1}{[1+( \gamma^{\prime}({\xi_1}))^2]^{1/2}}-\dfrac{1}{r},\quad &&\hbox{on $\Sigma$,} \\
\lim_{\xi_2\to\infty}\beta_2 &=0.\quad&&
\end{aligned}\right.
\end{equation}
Here $\triangle_{\xi}{:}=\pa_{\xi_1}^2+\pa_{\xi_2}^2$.
It is well-known that each problem has a unique solution,
and the solutions have the following decay properties~\cite[Theorem 2.2]{Friedman:1997}.
Similar results for Dirichlet boundary value problems can also be found in~\cite{AllaireAmar:1999,Mikelic:2009}.
\begin{lemma}\label{lem:estaux}
Let $\beta_0, \beta_1$ and $\beta_2$ be the solutions of~\eqref{eq:aux1},\eqref{eq:aux2} and~\eqref{eq:aux3}, respectively. Then, for $i=0,1$ and $2$, there exist constants $C$ and $\delta$ such that
\begin{equation}\label{eq:estaux}
\nm{\beta_i}{L^\infty(\tube)}+\nm{\na_{\xi}\beta_i}{L^\infty(\tube)}\le Ce^{-\delta\xi_2}.
\end{equation}
\end{lemma}
\subsection{The first order approximation}
Denote $\beta_i^\eps(x)=\beta_i(\xxe)$ for $i=0,1,2$, and define
\begin{equation}\label{eq:firstapp}
u^1(x)=\beta_0^\eps(x)+\beta_i^\eps\pa_{x_i}\hu(x).
\end{equation}
The first order approximation $u_1^\eps{:}=\hu+\eps u^1$, which
does not satisfy the homogeneous
Dirichlet boundary condition as $\uu$ on $\Ga_D$.
It is useful to introduce a corrector $\corr$ to the first order approximation, which satisfies $\corr-u^1\in V(\Om_0)$
and
\begin{equation}\label{eq:corrector}
a_{\Om_0}(\corr,v)=0\qquad \hbox{for all } v\in V(\Om_0).
\end{equation}

The corrector $\corr$ can be estimated as follows.
\begin{lemma}\label{lema:estcorr}
Let $\corr$ be the solution of~\eqref{eq:corrector},
then there exists $C$ that is
independent of the size of $\Omega_0$ such that
\begin{equation}\label{eq:estcorr}
\nm{\na\corr}{L^2(\Omega_0)}
 \le C\Lr{1+\nm{\na\hu}{L^\infty(\Om_0)}+\nm{\na^2\hu}{L^2(\Om_0)}}.
\end{equation}
\end{lemma}

\begin{proof}
Define a smooth cut-off function $\cut\in C_0^\infty(\Om_0)$ by
\[
\cut(x)=\left\{
\begin{aligned}
1,& \quad x\in\Om_0,\text{dist}(x,\Ga_D)\ge 2\eps,\\
0,&\quad x\in\Om_0,\text{dist}(x,\Ga_D)\le\eps,
\end{aligned}\right.
\]
and $\nm{\rho_\eps}{L^\infty(\Om_0)}\le 1$ and $\nm{\na \rho_\eps}{L^\infty(\Om_0)}\leq C/\eps.$

Let $\eta^{\eps}(x)=(1-\cut(x))u^1(x)$. It is clear to see
\[
\nm{\nabla\corr}{L^2(\Om_0)}\leq \nm{\na\eta^{\eps}}{L^2(\Om_0)}.
\]
A direct calculation gives
\[
\diff{\eta^{\eps}}{x_i}=-\pa_{x_i}\cut u^1+(1-\cut)\beta_j^\eps\pa_{x_ix_j}^2\hu
+(1-\cut)\Lr{\pa_{x_i}\beta_0^\eps+\pa_{x_i}\beta_j^\eps\pa_{x_j}\hu}.
\]

Using the decay estimate for $\beta_i^\eps$ in Lemma~\ref{lem:estaux}, we may bound $\nm{\na\eta^\eps}{L^2(\Om_0)}$ as follows. We only estimate the first term, other terms can be bounded similarly.
\begin{align*}
\nm{\pa_{x_i}\cut\beta_0^\eps}{L^2(\Om_0)}^2&\le C\eps^{-2}\Lr{\int_{\eps}^{2\eps}\!+\!\int_{1-2\eps}^{1-\eps}}\!\int_0^\infty \!\!e^{-2\delta x_2/\eps}\dx+C\eps^{-2}\!\!\int_0^1\!\int_{1-2\eps}^{1-\eps}e^{-2\delta x_2/\eps}\dx\\
&\le C\eps^{-2}\Lr{\eps^2+\eps e^{-2\delta/\eps}}\le C,
\end{align*}
where we have used~\eqref{eq:estaux} and the fact that $\cut$ supports in a narrow layer
of width $\mc{O}(\eps)$. Similarly, we have
\begin{align*}
\nm{\pa_{x_i}\rho_\eps\pa_{x_j}\hu\beta_j^\eps}{L^2(\Om_0)}
&\le C\nm{\na\hu}{L^\infty(\Om_0)},\\
\nm{(1-\cut)\pa_{x_ix_j}^2\hu\beta_j^\eps}{L^2(\Om_0)}&\le C\nm{\na^2\hu}{L^2(\Om_0)},\\
\nm{(1-\cut)\pa_{x_i}\beta_0^\eps}{L^2(\Om_0)}
&\le C,\\
\nm{(1-\cut)\pa_{x_j}\hu\pa_{\xi_i}\beta_j^\eps}{L^2(\Om_0)}
&\le C\nm{\na\hu}{L^\infty(\Om_0)}.
\end{align*}
Summing up all the terms, we obtain~\eqref{eq:estcorr} and complete the proof.
\end{proof}

The next theorem gives the error estimate for the first order approximation.
\begin{theorem}\label{theo:esthighorder}
Let $u^{\eps}$ and $u_0$ be the solutions of Problems~\eqref{eq:modelproblem} and~\eqref{eq:dirichlet0}, respectively.
Let $u^1$ be defined in~\eqref{eq:firstapp}. There exists $C$ independent of the size of $\Om_0$ such that
\begin{equation}\label{eq:esthighorder}
\nm{\nabla(\uu-u_1^\eps)}{L^2(\Oms)}
 \le C\eps\bigl(1+\nm{\na\hu}{W^{1,\infty}(\Om_0)}+\nm{\na^2\hu}{H^1(\Om_0)}\bigr).
\end{equation}
\end{theorem}

\begin{proof}
For any $v\in V(\Oms)$, an integration by parts yields
\[
\int_{\Oms}\na\hu\na v\dx=\int_{\Oms}fv\dx+\int_{\Ga_\eps}\diff{\hu}{n}v\dsx,
\]
which together with~\eqref{eq:weakform} gives
\begin{equation}\label{eq:step1}
\int_{\Oms}\na(\uu-\hu)\na v\dx=\int_{\Ga_\eps}\Lr{g_\eps-\diff{\hu}{n}}v\dsx.
\end{equation}

Next we calculate $\int_{\Oms}\na u^1\na v\dx$. Under the change of variables $\xi=x/\eps$, $\Oms$ is mapped
onto a domain
\[
D_{\eps,\xi}{:}=\set{\xi\in\R^2}{0<\xi_1<1/\eps,\gamma(\xi_1)<\xi_2<1/\eps}
\]
with the curved boundary \(\Ga_{\eps,\xi}{:}=\set{\xi\in\R^2}{0<\xi_1<1/\eps,\xi_2=\gamma(\xi_1)}\). We
denote $D_{0,\xi}$ as the mapped domain of $\Om_0$ under this map. Notice that
for any function $v$,
\[
D_x v=\na_x v+\dfrac1\eps\na_{\xi} v.
\]
Clearly,
\[
\int_{\Oms}\na u^1\na v\dx=\int_{\Oms}\na_x u^1\na v\dx+\int_{D_{\eps,\xi}}\na_{\xi} u^1\na_{\xi} v\md\xi.
\]
A direct calculation gives
\begin{align*}
\int_{D_{\eps,\xi}}\na_{\xi} u^1\na_{\xi} v\md\xi&=\int_{D_{\eps,\xi}}\Lr{\na_{\xi}\beta_0\na_\xi v+\na_\xi\beta_i\na_\xi\Lr{\diff{\hu}{x_i}v}}\md\xi\\
&\quad+\int_{D_{\eps,\xi}}\diff{}{x_i}(\na_\xi\hu)\Lr{\beta_i\na_\xi v-v\na_\xi\beta_i}\md\xi.
\end{align*}
Using the definition of $\{\beta_i\}_{i=0}^2$, an integration by parts yields
\[
\int_{D_{\eps,\xi}}\na_{\xi}\beta_0\na_\xi v\md\xi=\int_{\Ga_{\eps,\xi}}\Lr{g-\hg}v\dsxi,
\]
and
\begin{align*}
\int_{D_{\eps,\xi}}\na_\xi\beta_i\na_\xi\Lr{\diff{\hu}{x_i}v}\md\xi
&=-\int_{\Ga_{\eps,\xi}}\Lr{n_{\xi}^1\diff{\hu}{x_1}+n_{\xi}^2\diff{\hu}{x_2}}v\dsxi-\dfrac1{r}\int_{\Ga_{\eps,\xi}}\diff{\hu}{x_2}v\dsxi\\
&=-\dfrac1{\eps}\int_{\Ga_{\eps,\xi}}\diff{\hu}{n_{\xi}}v\dsxi-\dfrac{1}{r\eps}\int_{\Ga_{\eps,\xi}}\diff{\hu}{\xi_2}v\dsxi.
\end{align*}
Using the fact that $\pa\hu/\pa n=r\hg$ on $\Ga_0$, we rewrite the last term in the
right-hand side of the above identity as
\begin{align*}
-\dfrac{1}{r\eps}\int_{\Ga_{\eps,\xi}}\diff{\hu}{\xi_2}v\dsxi
&=-\dfrac{1}{r\eps}\int_{\Ga_{\eps,\xi}}\diff{\hu}{\xi_2}(\xi_1,0)v\dsxi\\
&\quad-\dfrac{1}{r\eps}\int_{\Ga_{\eps,\xi}}\Lr{\diff{\hu}{\xi_2}(\xi_1, \gamma(\xi_1))-\diff{\hu}{\xi_2}(\xi_1,0)}v\dsxi\\
&=\int_{\Ga_{\eps,\xi}}\hg v\dsxi
-\dfrac{1}{r\eps}\int_{D_{0,\xi}\backslash D_{\eps,\xi}}\dfrac{\pa^2\hu}{\pa\xi_2^2}\md\xi.
\end{align*}
Combining the above three equations, we obtain
\begin{align*}
\int_{D_{\eps,\xi}}&\na_{\xi} u^1\na_{\xi} v\md\xi=\int_{\Ga_{\eps,\xi}}
\Lr{g-\dfrac1{\eps}\diff{\hu}{n_{\xi}}}v\dsxi\\
&\quad+\int_{D_{\eps,\xi}}\diff{}{x_i}(\na_\xi\hu)\Lr{\beta_i\na_\xi v-v\na_\xi\beta_i}\md\xi
-\dfrac{1}{r\eps}\int_{D_{0,\xi}\backslash D_{\eps,\xi}}v\dfrac{\pa^2\hu}{\pa\xi_2^2}\md\xi.
\end{align*}
Denote $e=\uu-\hu-\eps u^1-\eps\corr$, we have the following error expansion:
\begin{align*}
\int_{\Oms}\na e\na v\dx&=\dfrac1{r}\int_{\Om_0\backslash\Oms}\dfrac{\pa^2\hu}{\pa x_2^2} v\dx
-\eps\int_{\Oms}\na_x u^1\na v\dx-\eps\int_{\Oms}\na\corr\na v\dx\\
&\quad-\int_{D_{\eps,\xi}}\diff{}{\xi_i}(\na_\xi\hu)\Lr{\beta_i\na_\xi v-v\na_\xi\beta_i}\md\xi.
\end{align*}
By Lemma~\ref{lema:strip}, we obtain
\[
\abs{\dfrac1{r}\int_{\Om_0\backslash\Oms}\dfrac{\pa^2\hu}{\pa x_2^2} v\dx}\le C\eps\nm{\na^2\hu}{H^1(\Om_0)}\nm{v}{H^1(\Om_0)}.
\]
The second term can be bounded as
\[
\abs{\eps\int_{\Oms}\na_x u^1\na v\dx}\le\eps\nm{\na_x u^1}{L^2(\Oms)}\nm{\na v}{L^2(\Oms)}
\le C\eps\nm{\na^2\hu}{L^2(\Oms)}\nm{\na v}{L^2(\Oms)},
\]
where $C$ depends on $\nm{\beta_i}{L^\infty}$, which are uniformly bounded.

We transform the last integrand back to $\Oms$ as
\[
\int_{D_{\eps,\xi}}\diff{}{\xi_i}(\na_\xi\hu)\Lr{\beta_i\na_\xi v-v\na_\xi\beta_i}\md\xi
=\eps\int_{\Oms}\diff{}{x_i}(\na_x\hu)\Lr{\beta_i\na_x v-v\na_x\beta_i}\dx.
\]
The first term can be bounded as
\[
\eps\abs{\int_{\Oms}\diff{}{x_i}(\na_x\hu)\beta_i\na_x v}\le\eps
\max_i\nm{\beta_i}{L^\infty}\nm{\na^2\hu}{L^2(\Oms)}\nm{\na v}{L^2(\Oms)}.
\]
By Lemma~\ref{lem:estaux}, we bound the second term as
\begin{align*}
\eps\abs{\int_{\Oms}\diff{}{x_i}(\na_x\hu)v\na_x\beta_i\dx}&\le C\nm{\na^2\hu}{L^\infty(\Oms)}\int_{\Oms}e^{-\delta x_2/\eps}\abs{v}\dx\\
&\le C\nm{\na^2\hu}{L^\infty(\Oms)}\int_0^\infty\!\! e^{-\delta x_2/\eps}\!\int_0^1\!\abs{v}\dx_1\!\dx_2.
\end{align*}
By trace inequality, for any $x_2\in (0,1)$, we have
\[
\int_0^1\abs{v(x_1,x_2)}\md x_1\le C\nm{v}{H^1(\Oms)}.
\]
Combining the above two inequalities, we obtain
\[
\eps\abs{\int_{\Oms}\diff{}{x_i}(\na_x\hu)v\na_x\beta_i\dx}\le  C\eps\nm{\na^2\hu}{L^\infty(\Oms)}\nm{v}{H^1(\Oms)}.
\]
Summing up all the estimates, we obtain that for any $v\in V_0(\Om_\eps)$,
\begin{align}\label{eq:laststep}
\abs{\int_{\Oms}\na e\na v\dx}&\le C\eps\Lr{\nm{\na^2\hu}{H^1(\Om_0)}+\nm{\na^2\hu}{L^\infty(\Oms)}+1}\nm{v}{H^1(\Om_0)}\nn\\
&\quad+\eps\nm{\na\corr}{L^2(\Oms)}\nm{\na v}{L^2(\Oms)}.
\end{align}
Since $\Ga_\eps$ is uniformly Lipschitz, we can extend $\uu$ from $\Oms$ to $\Om_0$ so that
\[
\nm{e}{H^1(\Om_0)}\le C\nm{e}{H^1(\Oms)},
\]
where $C$ only depends on $\nm{h'}{L^\infty(Y)}$ by Lemma~\ref{lema:extension}. Taking
$v=e$ in~\eqref{eq:laststep}, we obtain
\begin{align*}
\nm{\na e}{L^2(\Oms)}^2&\le C\eps\bigl(\nm{\na^2\hu}{H^1(\Om_0)}+\nm{\na^2\hu}{L^\infty(\Oms)}+1
+\nm{\na\corr}{L^2(\Oms)}\bigr)\nm{e}{H^1(\Oms)}.
\end{align*}
Using Poincare's inequality to $e$ because $e\in V_0(\Oms)$, we obtain
\begin{align*}
\nm{\na e}{L^2(\Oms)}&\le C\eps\bigl(\nm{\na^2\hu}{H^1(\Om_0)}+\nm{\na^2\hu}{L^\infty(\Oms)}+1
+\nm{\na\corr}{L^2(\Oms)}\bigr),
\end{align*}
which together with~\eqref{eq:estcorr} yields the desired estimate~\eqref{eq:esthighorder}.
\end{proof}
\section{Error Estimate}
We are ready to prove the convergence rate of the proposed MsFEM by the homogenization results
in the last section.
For any element
$\tau_\eps$ with a rough side on $\Gamma_{\eps}$, we assume that $\tau_\eps$ is contained in its
homogenized domain $\tau_0$. Given this assumption, we could apply Theorem~\ref{theo:esthighorder}
to each element. In fact, this seemingly restrictive assumption is not essential because Theorem~\ref{theo:esthighorder} remains valid without such assumption. Therefore, the error estimate
also holds true without this assumption, which is also confirmed by the numerical examples in the next section. In addition, to avoid too many technical complexity, the estimate is restricted to the triangular element, while
the proof can be generalized to the quadrilateral element with minor modifications.
\subsection{Homogenization of multiscale basis functions}
We start with some homogenization results of the multiscale
basis functions $\Phi_{p,\tau_{\eps}}^{\mr{MS}}$.
By the homogenization results in last section, we may clarify the
zeroth order approximation and the first order approximation of $\Phi_{p,\tau_{\eps}}^{\mr{MS}}$,
which are denoted by $\Phi_{p,\tau_{\eps}}^{\mr{0}}$ and $\Phi_{p,\tau_{\eps}}^{1}$,
respectively.

Note that $\Phi_{p,\tau_{\eps}}^{\mr{0}}-\phi_{p,\tau_0}\in V_0(\tau_0)$ and
\[
a_{\tau_0}(\Phi_{p,\tau_{\eps}}^{\mr{0}},v)=(r\aver{\theta}_{p,\tau_\eps},v)_{\pa\tau_0\cap \Ga_0} \qquad\text{for all } v\in V_0(\tau_0).
\]
It is clear
\(
r\aver{\theta}_{p,\tau_\eps}=\pa_n\phi_{p,\tau_0}.
\)
We conclude that the unique solution of the above problem is $\Phi_{p,\tau_{\eps}}^{\mr{0}}=\phi_{p,\tau_0}$.

The first order corrector $\Phi_{p,\tau_{\eps}}^{\mr{1}}$ of $\Phi_{p,\tau_{\eps}}^{\mr{MS}}$ is given by
\[
\Phi_{p,\tau_{\eps}}^{\mr{1}}=\wt{\beta}_0^\eps\pa_n\phi_{p,\tau_0}+\beta_i^\eps\diff{\phi_{p,\tau_0}}{x_i},
\]
where $\wt{\beta}_0^\eps(x)=\wt{\beta}_0(x/\eps)$ with $\wt{\beta}_0$ being the solution of
\[
\left\{
\begin{aligned}
-\tri_{\xi}\wt{\beta}_0&=0,&&\qquad\text{in\quad} Z_{bl}, \\
\pa_n\wt{\beta}_0&= \dfrac{1}{r}\Big(g(\xi_1)/\hg-1\Big),&&\qquad \text{on\;}\Sigma, \\
\lim_{\xi_2\to\infty} \wt{\beta}_0&=0.&&
\end{aligned}\right.
\]
It is clear that $\wt{\beta}_0= 0$ if $g=\hg$.
When $\wt{\beta}_0= 0$, the proof is simpler than the case $\wt{\beta}_0\neq 0$ but the estimate is same. We only consider the later case in the following.

For any $\tau\in\mc{T}_h$, we define the MsFEM interpolant of $\uu$ as
\[
\Pi_{h}\uu{:}=\sum_{p\in\ms{S}(\tau)}\hu(x_p)\uph.
\]
It is clear to see $\Pi_h\uu$ reduces to the standard linear Lagrange
interpolant of $\hu$, which is denoted by $\pi_h\hu$ when $\tau$ has no side on $\Ga_\eps$. For element
$\tau_\eps$ with a rough side, we define the first-order
approximation of the MsFEM interpolant by
\[
(\Pi_h\uu)_1{:}=\sum_{p\in\ms{S}(\tau_\eps)}\hu(x_p)\Lr{\phi_{p,\tau_0}+\eps\Phi_{p,\tau_\eps}^1}.
\]
The interpolate estimate is based on the following decomposition
\begin{equation}\label{eq:dec}
\uu-\Pi_h\uu=\Lr{\uu-\uu_1}+\Lr{\uu_1-(\Pi_h\uu)_1}+\Lr{(\Pi_h\uu)_1-\Pi_h\uu}.
\end{equation}

The following lemma is a direct consequence of Theorem~\ref{theo:esthighorder}.
%
\begin{lemma}\label{lem:homoMsFEM}
For any rough-sided element $\tau_{\eps}$, we have
\begin{equation}\label{eq:estbasis1}
\nm{\na\Pi_h\uu-\na(\Pi_h\uu)_1}{L^2(\tau_\eps)}\le C\eps(1+\nm{\na\hu}{L^\infty(\tau_0)}).
\end{equation}
\end{lemma}

By definition, we rewrite $(\Pi_h\uu)_1$ as
\begin{equation}\label{eq:interobservation}
(\Pi_h\uu)_1=\pi_h\hu+\eps\wt{\beta}_0^\eps\diff{\pi_h\hu}{n}+\eps\beta_i^\eps\diff{\pi_h\hu}{x_i},
\end{equation}
and
\begin{equation}\label{eq:ander1st}
\uu_1=\hu+\eps\wt{\beta}_0^\eps\diff{\hu}{n}+\eps\beta_i^\eps\diff{\hu}{x_i}.
\end{equation}

A direct consequence of the representations~\eqref{eq:interobservation} and~\eqref{eq:ander1st}
is
\begin{lemma}\label{lem:est1stTerm}
For any rough-sided element $\tau_{\eps}$,  we have
\[
\nm{\na\uu_1-\na(\Pi_h\uu)_1}{L^2(\tau_\eps)}\le C
\Lr{h_{\tau_\eps}+\eps}\nm{D^2\hu}{L^2(\tau_\eps)}.
\]
\end{lemma}

\begin{proof}
A direct calculation gives that for $i=1,2$,
\begin{align*}
\diff{}{x_i}\Lr{\uu_1-\na(\Pi_h\uu)_1}&=\diff{}{x_i}(\hu-\pi_h\hu)+\eps\diff{\wt{\beta}_0^\eps}{x_i}\diff{}{n}(\hu-\pi_h\hu)\\
&\quad+\eps\wt{\beta}_0^\eps\dfrac{\pa^2\hu}{\pa n\pa x_i}
+\eps\diff{  {\beta}_j^\eps}{x_i}\diff{}{x_j}(\hu-\pi_h\hu)
+\eps  {\beta}_j^\eps\dfrac{\pa^2\hu}{\pa x_i\pa x_j}.
\end{align*}
Note that $\wt{\beta}_0$ satisfies the same decay estimate~\eqref{eq:estaux} as $\beta_0$, and proceeding
along the same line that leads to Lemma~\ref{lema:estcorr}, we obtain
\begin{align*}
\int_{\tau_\eps}\abs{\na\wt{\beta}_0^\eps\pa_n(\hu-\pi_h\hu)}^2\dx&\le C\eps^{-2}\int_{\tau_\eps\cap\Ga_\eps}\abs{\na(\hu-\pi_h\hu)}^2\dx_1\int_0^h e^{-2\delta x_2/\eps}\dx_2\\
&\le C\eps^{-1}\int_{\tau_\eps\cap\Ga_\eps}\abs{\na(\hu-\pi_h\hu)(x_1,x_2)}^2\dx_1.
\end{align*}
By the trace inequality, we get
\begin{align*}
\nm{\na\wt{\beta}_0^\eps\pa_n(\hu-\pi_h\hu)}{L^2(\tau_\eps)}&\le C\eps^{-1/2}\nm{\na(\hu-\pi_h\hu)}{L^2(\tau_\eps)}
^{1/2}\nm{\na^2(\hu-\pi_h\hu)}{L^2(\tau_\eps)}^{1/2}\\
&\le C(h_{\tau_\eps}/\eps)^{1/2}\nm{\na^2\hu}{L^2(\tau_\eps)}.
\end{align*}

Proceeding along the same line that leads to the above inequality, we obtain
\[
\nm{\na\beta_i^\eps\pa_{x_i}(\hu-\pi_h\hu)}{L^2(\tau_\eps)}\le C(h_{\tau_\eps}/\eps)^{1/2}\nm{\na^2\hu}{L^2(\tau_\eps)}.
\]
The remaining terms may be bounded as follows.
\[
\nm{\na(\hu-\pi_h\hu)}{L^2(\tau_\eps)}\le Ch_{\tau_\eps}\nm{\na^2\hu}{L^2(\tau_\eps)},
\]
and
\[
\nm{\wt{\beta}_0^\eps\pa_n(\na\hu)}{L^2(\tau_\eps)}\le C\nm{\na^2\hu}{L^2(\tau_\eps)},\quad
\nm{\beta_i^\eps\pa_{x_i}(\na\hu)}{L^2(\tau_\eps)}\le C\nm{\na^2\hu}{L^2(\tau_\eps)}.
\]
Combining the above estimates, we obtain
the desired estimate.
\end{proof}

For any element $\tau$ without rough edge, $\Pi_h\uu=\pi_h\hu$. The decomposition~\eqref{eq:dec} is replaced by $\uu-\Pi_h\uu=\uu-\uu_1+\hu-\pi_h\hu+\eps u^1$. Therefore, we need the a priori estimate for $u^1$ over elements without rough edge. We divide the elements into three groups; the elements with one rough edge belong to $\mc{T}_h^1$, the elements with one vertex on the rough boundary belong to $\mc{T}_h^2$, and
the remaining elements belong to $\mc{T}_h^3$.
\begin{lemma}\label{lem:estU1}
When $\tau\in\mc{T}_h^2$, then
\begin{equation}\label{eq:estU1b}
\begin{aligned}
\sum_{\tau\in\mc{T}_h^2}\nm{\na u^1}{L^2(\tau)}^2&\le Ch^{-1}\Lr{1
+\nm{\na\hu}{L^{\infty}(\Oms)}^2}+C\sum_{\tau\in\mc{T}_h^2}\nm{\na^2\hu}{L^2(\tau)}^2.
\end{aligned}
\end{equation}

When $\tau\in\mc{T}_h^3$, then
\begin{equation}\label{eq:estU1a}
\begin{aligned}
\sum_{\tau\in\mc{T}_h^3}\nm{\na u^1}{L^2(\tau)}^2&\le C\dfrac{h}\eps\Lr{1+\nm{\na\hu}{W^{1,\infty}(\Oms)}^2}+C\sum_{\tau\in\mc{T}_h^3}\nm{\na^2\hu}{L^2(\tau)}^2.
\end{aligned}
\end{equation}
\end{lemma}

\begin{proof}
For an element $\tau\in\mc{T}_h^2$, we assume that the two sides intersect with $\Ga_\eps$ are given explicitly by
$x_1=\alpha_1 x_2$ and $x_1=\alpha_2 x_2$, with $|\alpha_i|\leq c_1$ and the bound $c_1$ depends only on the minimal
angle of $\tau$. Using~\eqref{eq:estaux}, a direct calculation gives
\begin{align*}
\nm{\na \beta_0^{\eps}}{L^2(\tau)}\!
\le C\eps^{-1}\!\Lr{\int_{\tau}e^{-{2\delta x_2}/{\eps}}\mr d\mr x}^{1/2} \!\!\!\! \le \! C \eps^{-1}\!\Lr{\!\int_{0}^\infty\!\!(\alpha_2-\alpha_1)x_2e^{-{2\delta x_2}/{\eps}}\dx_2}^{1/2}
\!\!\! \le  C .
\end{align*}

A direct calculation gives that for $i=1,2$, there holds
\[
\nm{\na(\partial_{x_i} \hu\beta_i^\eps)}{L^2(\tau)}\le
C\Lr{\nm{\na\hu}{L^{\infty}(\tau)}+\nm{\na^2\hu}{L^2(\tau)}}.
\]
Combining the above estimates and using the fact that the cardinality of $\mc{T}_h^2$ is $\mc{O}(h^{-1})$, we obtain~\eqref{eq:estU1b}.

Using the facts that the triangulation is regular, for the element in $k$-th layer, there exists a constant $c_0$ such that $c_0kh\le\text{dist}(\tau,\Ga_\eps)\le c_0(k+1)h$. By~\eqref{eq:estaux}, a direct calculation gives
\begin{align*}
\nm{\na \beta_0^{\eps}}{L^2(\tau)}&\leq C\eps^{-1}\Lr{\int_{\tau}e^{-{2\delta x_2}/{\eps}}\mr d\mr x}^{1/2}\!\!\!\le C h_{\tau}^{1/2}\eps^{-1}\Lr{\int_{c_0kh}^{c_0(k+1)h}\!\!\!e^{-{2\delta x_2}/{\eps}}\dx_2}^{{1}/{2}}\\
&\leq C(h/\eps)^{1/2}\exp(-c_0\delta kh/\eps).
\end{align*}
Proceeding along the same line that leads to the above estimate, we have for $i=1,2$,
\[
\nm{\na(\pa_{x_i}\hu\beta_i^\eps)}{L^2(\tau)}\le
C(h_\tau/\eps)^{1/2}\exp(-c_0\delta kh/\eps)\nm{\na\hu}{L^\infty(\tau)}
+C\nm{\na^2\hu}{L^2(\tau)}
\]
A combination of the above estimates leads to
\begin{align*}
\nm{\na u^1}{L^2(\tau)}&\le C\exp{(-c_0\delta kh/\eps)}(h_{\tau}/\eps)^{1/2}\Bigl(1+\nm{\na\hu}{L^\infty(\tau)}\Bigr)+C\nm{\na^2\hu}{L^2(\tau)}.%
\end{align*}
Summing up all the elements in $\mc{T}_h^3$ leads to~\eqref{eq:estU1a}.
\end{proof}
\begin{figure}[htbp]\vspace{-0.25cm}
  \begin{center}
  \resizebox{!}{4.0cm}
  {\includegraphics[height=6.0cm]{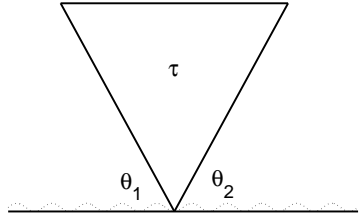}}
  \end{center}
\caption{An element with a vertex on the rough boundary.}
\label{fig:element}
\end{figure}
\subsection{Interpolation error estimate}
The next theorem gives the interpolate error estimate. 
\begin{theorem}\label{theo:MsFEMInterp}
Let $\uu$ be the solution of the problem~\eqref{eq:modelproblem}, we have
\begin{align}\label{eq:estinterp1}
\nm{\na(\uu-\Pi\uu)}{L^2(\Oms)}&
\le C\eps\Lr{\nm{\na^2\hu}{H^1(\Om_0)}+\nm{\na\hu}{W^{1,\infty}(\Om_0)}}+Ch\nm{\na^2\hu}{L^2(\Om_0)}\nn\\
&\quad+C\eps h^{-{1}/{2}}(\nm{\na\hu}{L^\infty(\Om_0)}+1).
\end{align}
\end{theorem}

\begin{proof}
We start from the following decomposition~\eqref{eq:dec}. Using Lemma~\ref{lem:est1stTerm},
\[
\sum_{\tau\in\mc{T}_h^1}
\nm{\na\uu_1-\na(\Pi_h\uu)_1}{L^2(\tau)}^2\le C(\eps+h)^2\sum_{\tau\in\mc{T}_h^1}
\nm{\na^2\hu}{L^2(\tau)}^2.
\]

For $\tau\in\mc{T}_h^2$, we have $\uu_1-(\Pi_h\uu)_1=\hu-\pi_h\hu+\eps u^1$. Therefore,
\begin{align*}
\sum_{\tau\in\mc{T}_h^2}\nm{\na\uu_1-\na(\Pi_h\uu)_1}{L^2(\tau)}^2
&\le 2\sum_{\tau\in\mc{T}_h^2}\nm{\na(\hu-\pi_h\hu)}{L^2(\tau)}^2
+2\eps^2\sum_{\tau\in\mc{T}_h^2}\nm{\na u^1}{L^2(\tau)}^2.
\end{align*}
Using Lemma~\ref{lem:estU1}, we obtain
\begin{align*}
\sum_{\tau\in\mc{T}_h^2}\!\nm{\na\uu_1-\na(\Pi_h\uu)_1}{L^2(\tau)}^2
\!\le\! C(\eps+h)^2\!\!\sum_{\tau\in\mc{T}_h^2}\!\nm{\na^2\hu}{L^2(\tau)}\!\!+C\dfrac{\eps^2}h\!\Lr{1+\nm{\na\hu}{L^\infty(\tau)}^2}.
\end{align*}

Proceeding along the same line that leads to the above estimate, we obtain
\begin{align*}
\sum_{\tau\in\mc{T}_h^3}\!
\nm{\na\uu_1-\na(\Pi_h\uu)_1}{L^2(\tau)}^2\!\le\! C(\eps+h)^2\!\!\sum_{\tau\in\mc{T}_h^3}\!\nm{\na^2\hu}{L^2(\tau)}^2\!\!+C\eps h\!\Lr{1+\nm{\na\hu}{L^\infty(\Oms)}^2}.
\end{align*}
Summing up all the above estimates, we obtain
\begin{align*}
\nm{\na\uu_1-\na(\Pi_h\uu)_1}{L^2(\Oms)}&\le C(\eps+h)\nm{\na^2\hu}{L^2(\Oms)}+C\eps h^{-1/2}\Lr{1+\nm{\na\hu}{L^\infty(\Oms)}}.
\end{align*}

By Lemma~\ref{lem:homoMsFEM},
\begin{align*}
\nm{\na\Pi_h\uu-\na(\Pi_h\uu)_1}{L^2(\Oms)}^2&=\sum_{\tau\in\mc{T}_h^1}\nm{\na\Pi_h\uu-\na(\Pi_h\uu)_1}{L^2(\tau_\eps)}^2\\
&\le C\eps^2\sum_{\tau\in\mc{T}_h^1}\nm{\na\hu}{L^\infty(\tau_0)}
\le C\dfrac{\eps^2}h\nm{\na\hu}{L^\infty(\Om_0)}.
\end{align*}

Finally, the term $\nm{\na(\uu-\uu_1)}{L^2(\Oms)}$ can be bounded by Theorem~\ref{theo:esthighorder}. Summing up all
the estimates we obtain the desired estimate~\eqref{eq:estinterp1}.
\end{proof}

Using the above interpolation estimate, we obtain the main result of this paper.
\begin{theorem}
Let $\uu$ and $u_h$ be the solutions of Problem~\eqref{eq:modelproblem} and Problem \eqref{eq:MsFE}, respectively. Then
\begin{align}\label{eq:errest}
\nm{\na(\uu-u_h)}{L^2(\Oms)}&
\le C\eps\Lr{\nm{\na^2\hu}{H^1(\Om_0)}+\nm{\na\hu}{W^{1,\infty}(\Om_0)}}+Ch\nm{\na^2\hu}{L^2(\Om_0)}\nn\\
&\quad+C\eps h^{-{1}/{2}}(\nm{\na\hu}{L^\infty(\Om_0)}+1).
\end{align}
\end{theorem}
\begin{remark}
We have not estimated the $L^2$ error of the method, because the $H^1$ error estimate is not optimal
with respect to the regularity of the data. Standard dual argument only yields a suboptimal convergence rate
as the original MsFEM~\cite{HouWuCai:1999}. This would be a topic for further study.
\end{remark}
%
\section{Numerical Examples}
In this section, we perform three numerical experiments to verify the convergence rate and efficiency of the
proposed method. We solve Problem~\eqref{eq:modelproblem} for different rough domains, different source terms
and different boundary fluxes.

\subsection{Implementation}
The implementation of the method is similar to the standard MsFEM~\cite{HouWu:1997, HouWuCai:1999}. The cell problem~\eqref{eq:cellpb} is numerically solved only for elements with a rough side,  and we use $P_1$ element to solve~\eqref{eq:cellpb} with the subgrid mesh size around $\eps/20$ in the simulations below. 

As an example, we consider a  square domain with a rough boundary. The more general case can be done similarly.
The domain $\Omega^\eps$ is given in~\eqref{eq:square} as in Figure~\ref{fig:domain4}. The function $\gamma_{\eps}\simeq\mc{O}(\eps)$ that represents the curved boundary will be specified in the examples.
We partition $\Oms$ with a uniform triangular mesh as in Figure~\ref{fig:partition}(b). For a given number $N$, we set $h=1/N$. The vertexes far away from the rough boundary are given by
$x_{i,j}=\{(ih, jh)\}$, where $i=0,\cdots N, j=1,\cdots, N$; and the vertexes on the rough boundary are given by $x_{i,0}=(ih, \gamma_{\eps}(ih))$, $i=0,\cdots N$.

The basis function is constructed as follows. For elements without rough boundary,  the basis function coincides with the standard linear basis function. For elements with a rough boundary,
we compute the multiscale basis function for an element
\[
e=\set{x\in\R^2}{0<x_1<h,\gamma_{\eps}(x_1)<x_2<\gamma_{\eps}(0)+(1-\gamma_{\eps}(0)/h)x_1}
\]
for example. Here we actually assume that $\gamma_{\eps}(x_1)<\gamma_{\eps}(0)+(1-\frac{\gamma_{\eps}(0)}{h})x_1$ holds for all $x_1\in(0,h)$. In reality, we can always make it true by choosing proper vertexes in the triangulation near the rough boundary. We firstly rescale $e$ and solve cell problem \eqref{eq:cellpb} by linear finite element over the rescaled element
\[
\hat{e}=\set{\wh{x}\in\R^2}{0<\hat{x}_1<1, h^{-1}\gamma_\eps(h\hat{x}_1)<\hat{x}_2<\gamma_\eps(0)/h +(1-\gamma_{\eps}(0)/h)\hat{x}_1}
\]
with a rough boundary $\hat{\Gamma}_e=\set{\wh{x}\in\R^2}{0<\hat{x}_1<1,\hat{x}_2=h^{-1}\gamma_\eps(h\hat{x}_1)}$.
We triangulate $\wh{e}$ by  a subgrid with maximum mesh size $\wt{h}$; see, e.g., $\tilde{h}\leq \eps/20$.  $\hat{\Gamma}_e$ is approximated by $\hat{\Gamma}_{e,\tilde{h}}$. For $g_\eps$ whose average on $\hat{\Gamma}_{e,\tilde{h}}$ is given by
\[
\aver{g_\eps}=\int_{\hat{\Gamma}_{e,\tilde{h}}}g_\eps(x_1/h)\ds/r\quad\text{with\quad}
r=\abs{\hat{\Gamma}_{e,\tilde{h}}}.
\]
The flux is given by
\[
\hat{\theta}_i(\hat{x}_1)=\left\{
\begin{aligned}
b_i/r \quad & \hbox{if } \|g_\eps-\aver{g_\eps}\|_{L^{\infty}(\hat{\Gamma}_{e,\tilde{h}})}<\eps,\\
\dfrac{b_i}{r}\dfrac{g_\eps}{\aver{g_\eps}}\quad  & \hbox{otherwise,}
\end{aligned}
\right.
\]
where $b_1=0$, $b_2=1$ and $b_3=-1$.
%


\subsection{Numerical experiments}
\textbf{First Example\;} In this example, the rough domain is given by
\[
\Omega_{\eps}= \set{x\in\R^2}{0<x_1<1,\eps \gamma(x_1/\eps)<x_2<1},
\]
where $\ga(y_1)=(\cos(2\pi y_1)-1)/10$. The rough boundary is given by
\[
\Gamma_\eps=\set{x\in\R^2}{0<x_1<1,x_2=\eps \gamma({x_1}/{\eps})},
\]
and $\Gamma_D=\partial\Omega_\eps\setminus\Gamma_\eps$.
We choose $f=1$ and $g=0$ in~\eqref{eq:modelproblem}, and set homogeneous Dirichlet
boundary condition on $\Ga_D$ and $\eps={1}/{128}$.

The grid in this example is the uniform triangular grid as the right subfigure in Fig.~\ref{fig:partition} with the mesh size $h$ varying from ${1}/{5}$ to ${1}/{160}$.
The errors are measured by
\[
\text{Err}_{L2}=\nm{u_h-\tilde{u}}{L^2(\Oms)}, \qquad\text{Err}_{H1}=\nm{\nabla u_h-\nabla \tilde{u}}{L^2(\Oms)}.
\]
Here $\tilde u$ is a solution computed on an adaptive refined mesh with mesh size $h\approx 10^{-4}$ by linear finite element.

Fig.~\ref{fig:4.1} shows two different scenarios of the convergence behaviour of the method. When the mesh size $h$
is larger than the roughness parameter $\eps$, the method is first order in the $H^1$ semi-norm and second order
in the $L^2$ norm. When $h$ is commensurate with $\eps$, the method degenerates due to the resonance error. This is consistent with our theoretical prediction, at least for the $H^1$ error.
\begin{figure}[htb!]
\hspace{-0cm}
\resizebox{!}{7.5cm} 
    {\includegraphics[height=7.5cm]{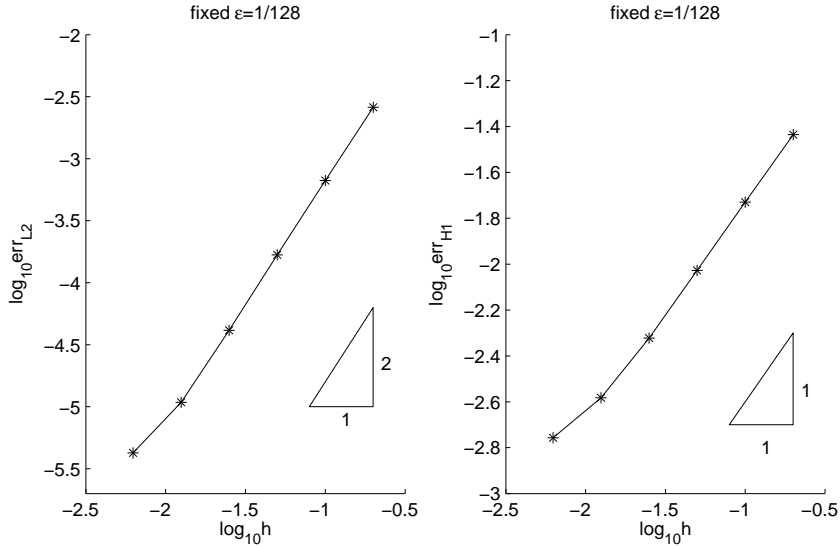}}
\caption{Convergence behavior of the method for the first example.}
\label{fig:4.1}
\end{figure}

\textbf{Second example\;}In this example, the domain $\Oms$ is the same with that in the first example. Unlike
the first example, we choose $f=0$ and an inhomogeneous boundary flux \(g_\eps=\Lr{1-\cos(2\pi x_1/\eps)}/2\).
In addition, we impose an inhomogeneous Dirichlet boundary value $u(x)=(1-x_2)/{2}$ on $\Gamma_D$.

Uniform triangular grid with the mesh size $h$ varying from $1/5$ to $1/80$ is employed in this example.
Fig.~\ref{fig:4.3} shows the convergence behavior of the method.
It is clear that the method has nearly optimal convergence rate for both the
$H^1$ semi-norm and the $L^2$ norm when $h=1/40>2\eps$, while the resonance error gets to dominate and the convergence rate degenerates when $h$ is approximately $1/80$.

In this example, we use linear finite element to solve the homogenized problem
\[
\left\{
\begin{aligned}
-\Delta\hu&=0\quad                 &&\hbox {in }\Om_0,\\
\hu&=\dfrac{1-x_2}{2}\quad  &&\hbox {on } \Ga_D,\\
\diff{\hu}{n}&= \frac{r}{2} \quad  &&\hbox {on } \Ga_0.
\end{aligned}\right.
\]
where $\Om_0=\{x\in\mathbb{R}^2 : 0<x_1<1, 0<x_2<1\}$, $\Gamma_0=\{(x_1,0):0<x_1<1\}$ and  $r=\int_{0}^1[1+(\ga'(y_1))^2]^{1/2} \dy_1 \approx 1.01$. This problem is solved by linear finite element
with a uniform triangulation of the homogenized domain $\Omega_0$.
The numerical solutions is
denoted by $u^0_h$. We compute the following quantities
\[
\text{err}_{L2}=\nm{\hu_h-\tilde{u}}{L^2(\Omega_0)},\quad\text{and}\quad
\text{err}_{H1}=\nm{\nabla \hu_h-\nabla \tilde{u}}{L^2(\Omega_0)},
\]
which is reported in Figure~\ref{fig:4.3}. It seems that the linear finite element method is less accurate as MsFEM.
This is due to the fact that the homogenization errors dominate as the
mesh is refined. This degeneracy of the convergence rate is more significant for the $L^2$ error.
\begin{figure}[htb!]
\hspace{-0cm}
\resizebox{!}{7.5cm} 
    {\includegraphics[height=7.5cm,width=12cm]{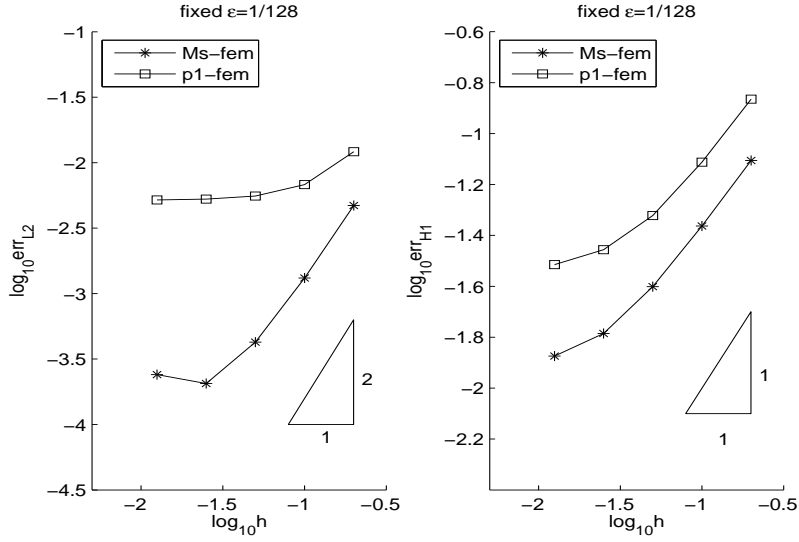}}
\caption{Convergence behavior of the method and standard $P_1$-element
 approximation for the second example.}
\label{fig:4.3}
\end{figure}

\textbf{Third Example\;}In this example, we test the problem with a non-periodic rough boundary, which is not covered by
our theoretical results, while the method works as well.
The domain is
\[
\Oms=\set{x\in\R^2}{0<x_1<1,\dfrac{\eps}{10}(\gamma(x_1)-1)<x_2<1}
\]
with $\gamma$ an oscillating function defined as follows.
We firstly divide the interval $(0,1)$ uniformly as
\(0=s_0<s_1<\cdots<s_M=1\) with $M={1}/{\eps}=128$.  The function $\gamma$ is set to be a piecewise continuous linear function over such  partition, with $\gamma(s_i)$, $i=0,\cdots,M$, a series of pseudo-random numbers between $0$ and $1$ generated by a standard C++ library function. The rough boundary
\[
\Gamma_{\eps}=\set{x\in\R^2}{0<x_1<1,x_2={\eps}(\gamma(x_1)-1)/{10}}.
\]
We choose $f=1,g=0$, and impose a homogeneous Dirichlet boundary condition on $\Ga_D$.
We use a uniform triangular grid with the mesh size $h$ varies from $1/5$ to $1/160$.
The results for MsFEM is reported in Fig.~\ref{fig:4.2}. Similar to the
previous examples, we get optimal convergence rate when $h>2\eps$.
The resonance errors becomes dominate for the $L^2$ error when $h=1/80$, but still small for the $H^1$ error even when $h=1/160$.
This might indicate the resonance error for the $L^2$ error is more pronounced.
\begin{figure}[htb!]
\hspace{-0cm}
\resizebox{!}{7.5cm} 
    {\includegraphics[height=7.5cm,width=12cm]{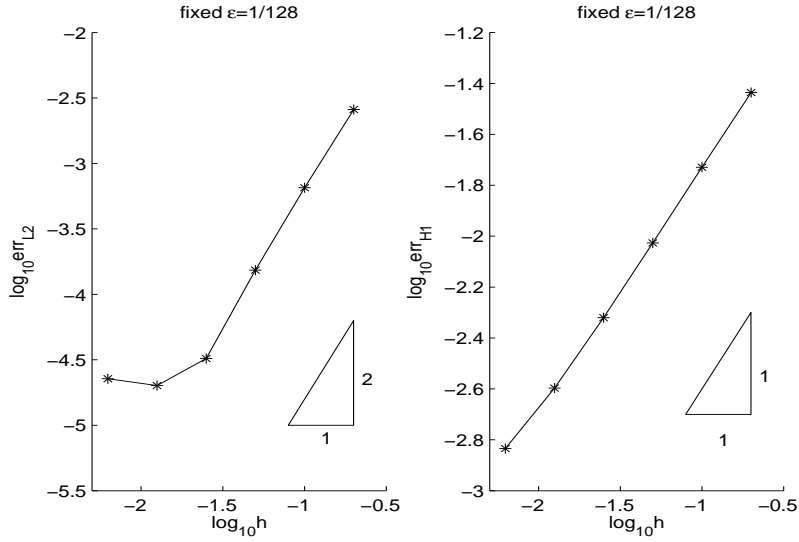}}
\caption{Convergence behavior of the method for the third example.}
\label{fig:4.2}
\end{figure}

\textbf{Fourth Example\;} In this example, we test the problem with a discontinuous right-hand side function $f$ as well as a non-periodic rough boundary, which is slightly more general than that in the previous example.
The domain is
\[
\Oms=\set{x\in\R^2}{0<x_1<1, \eps(\gamma(x_1)-1)<x_2<1}
\]
with $\gamma$ an oscillating function defined as follows.
We firstly divide the interval $(0,1)$ by
\(0=s_0<s_1<\cdots<s_M=1\) with $M={1}/{\eps}=128$. In comparison with the third example,
here $s_1,\cdots, s_{M-1}\in (0,1)$ are chosen randomly. For that purpose, we generate a series of pseudo-random numbers between $0$ and $1$ generated by a standard C++ library function. Then we sort
them in a sequence and denoted as $\{s_i\}$.
The function $\gamma$ is set to be a piecewise continuous linear function over such  partition, with $\gamma(s_i)$, $i=0,\cdots,M$, also a series of pseudo-random numbers between $0$ and $1$. The rough boundary
\[
\Gamma_{\eps}=\set{x\in\R^2}{0<x_1<1,x_2={\eps}(\gamma(x_1)-1)}.
\]
We choose $g=0$, and impose a homogeneous Dirichlet boundary condition on $\Ga_D$.
In addition, we choose a discontinuous right hand side function as
\begin{equation*}
f(x)=\left\{
\begin{array}{ll}
1, & \hbox{if } x_1<\frac{1}{2};\\
-1, & \hbox{if } x_1\geq \frac{1}{2}.
\end{array}
\right.
\end{equation*}
In this  case, the homogenized solution $u_0\in H^2(\Omega_0)$ and but not in $H^3(\Omega_0)$, i.e., it does not satisfy
the regularity assumption in our theoretical analysis.

We use a triangular grid with the mesh size $h$ varies from $1/5$ to $1/80$.
The results for MsFEM is reported in Fig.~\ref{fig:4.4}. Similar to the
previous examples, we get optimal convergence rate when $h>2\eps$.
\begin{figure}[htb!]
\hspace{0cm}
\resizebox{!}{7.5cm} 
    {\includegraphics[height=7.5cm,width=12cm]{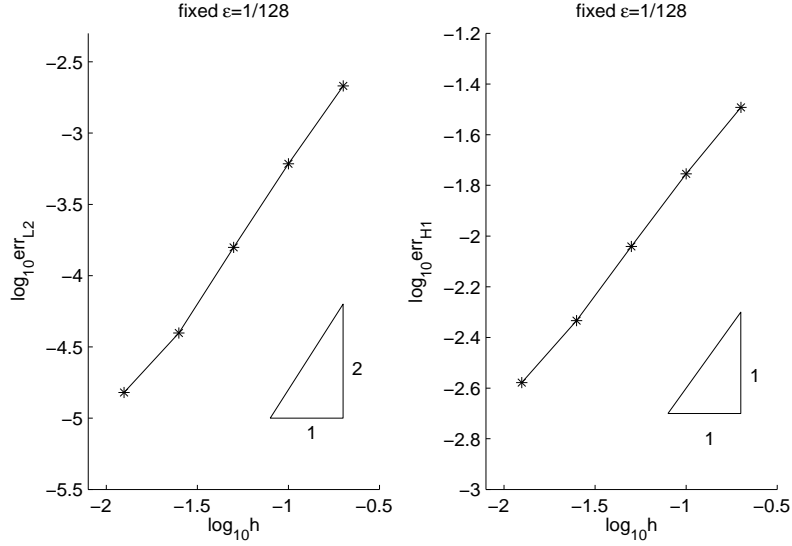}}
\caption{Convergence behavior of the method for the fourth example.}
\label{fig:4.4}
\end{figure}
We also calculate the 2-norm condition number of the resulting linear system. As shown in Fig.~\ref{fig:4.5},
the condition number scales as $\mc{O}(h^{-2})$, which is optimal and is independent of
the microstructure.
\begin{figure}[htb!]
\hspace{2cm}
\resizebox{!}{7.5cm} 
    {\includegraphics[height=5.5cm,width=6cm]{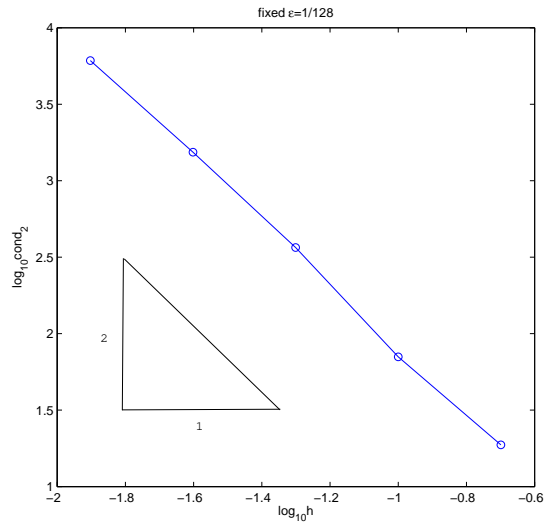}}
\caption{Condition numbers of the MsFEM system.}
\label{fig:4.5}
\end{figure}

{\centering \bf Acknowledgement.}
We would like to thank the anonymous referees for their many suggestions, that
help us to improve the paper.
\bigskip
\bibliographystyle{amsplain}
\bibliography{rgbd}
\end{document}